\documentclass[12pt]{amsproc}
\usepackage[left=1in, right=1in, top=1in, bottom=1in]{geometry}
\usepackage[authoryear]{natbib} 
\usepackage{amsmath,amsthm,amsfonts,amssymb,amsthm}
\usepackage[T1]{fontenc}
\usepackage{mathtools}
\usepackage[dvipsnames]{xcolor}
\usepackage[colorlinks=true,citecolor=blue,urlcolor=blue, linkcolor=blue, urlcolor=black]{hyperref}
\usepackage{blkarray}
\usepackage{tikz,tikz-cd}
\usepackage{cleveref}
\usepackage{blkarray}
\usetikzlibrary{chains,matrix,scopes,fit,decorations,positioning,shapes,arrows,decorations.markings}
\usepackage{subcaption}

\newcommand\C{\mathbb{C}}

\newcommand\cI{\mathcal{I}}

\newcommand\cS{\mathcal{S}}
\newcommand\cT{\mathcal{T}}

\newcommand\codim{\operatorname{codim}}
\newcommand\rO{\mathrm{O}}
\newcommand\rSO{\mathrm{SO}}
\newcommand\rSP{\mathrm{SP}}
\newcommand\T{\mathcal{T}}
\newcommand\Q{\mathbf{Q}}
\newcommand\bP{\mathbf{P}}
\newcommand\bR{\mathbf{R}}
\newcommand\bi{\mathbf{i}}
\newcommand\bj{\mathbf{j}}
\newcommand\ri{\mathrm{i}}
\newcommand\sym{\mathrm{sym}}
\newcommand\rank{\mathrm{rank}}
\newcommand\dist{\mathrm{dist}}

\newcommand\bs{\boldsymbol}

\newcommand\R{\mathbb{R}}

\newcommand\<{\langle}
\renewcommand\>{\rangle}

\newcommand{\floor}[1]{\lfloor #1 \rfloor}

\DeclareMathOperator{\diag}{diag}

\DeclareMathOperator{\Res}{Res}

\theoremstyle{plain}
\newtheorem{thm}{thm}[section]
\newtheorem{theorem}[thm]{Theorem}
\newtheorem{proposition}[thm]{Proposition}
\newtheorem{lemma}[thm]{Lemma}
\newtheorem{corollary}[thm]{Corollary}

\theoremstyle{definition}
\newtheorem{definition}[thm]
{Definition}
\newtheorem{example}[thm]
{Example}

\theoremstyle{remark}
\newtheorem{remark}[thm]
{Remark}

\title{Orthogonal eigenvectors and singular vectors of tensors}
\author{\'Alvaro Ribot}
\address{\'{A}lvaro Ribot, Harvard University}
\email{aribotbarrado@g.harvard.edu}
\author{Anna Seigal}
\address{Anna Seigal, Harvard University}
\email{aseigal@seas.harvard.edu}
\author{Piotr Zwiernik}
\address{
Piotr Zwiernik, 
University of Toronto, Universitat Pompeu Fabra, and Barcelona School of Economics }\email{piotr.zwiernik@upf.edu}

\begin{document}

\begin{abstract}
The spectral theorem says that a real symmetric matrix has an orthogonal basis of eigenvectors and that, for a matrix with distinct eigenvalues, the basis is unique (up to signs). In this paper, we study the symmetric tensors with an orthogonal basis of eigenvectors and show that, for a generic such tensor, the orthogonal basis is unique. This resolves a conjecture by Mesters and Zwiernik. We also study the non-symmetric setting.
The singular value decomposition says that a real matrix has an orthogonal basis of singular vector pairs
and that, for a matrix with distinct singular values, the basis is unique (up to signs). We describe the tensors with an orthogonal basis of singular vectors and show that a generic such tensor has a unique orthogonal basis, with one exceptional format: order four binary tensors.
We use these results to propose a new tensor decomposition that generalizes an orthogonally decomposable decomposition and specializes the Tucker decomposition.
\end{abstract}

\maketitle

\section{Introduction}

The spectral theory of symmetric matrices is a cornerstone of linear algebra, underpinning applications across mathematics, physics, and engineering. 
It says that a real $n\times n$ symmetric matrix $M$ has an orthogonal basis $v_1,\ldots,v_n$ of $\R^n$ such that each $v_i$ is a unit eigenvector of~$M$. If $M$ is generic, more precisely when all its eigenvalues are distinct, this basis is unique up to the signs of the basis vectors.
This result enables efficient diagonalization, quadratic form analysis, dimensionality reduction, and numerous other applications.

The question of extending this result to tensors has attracted considerable interest over the past two decades. The concepts of eigenvectors and eigenvalues were extended to tensors by \cite{lim2005singular} and \cite{qi2005eigenvalues}.  
This laid the groundwork for subsequent work to explore the spectral properties of symmetric tensors and their applications; see \cite{qi2018tensor} for an overview.

Unlike for matrices, a dimension count reveals that a general symmetric tensor does not have an orthogonal basis of eigenvectors. This raises questions about the conditions under which an orthogonal basis of eigenvectors exists, when the basis is unique, and the structure that the basis imposes in tensor decomposition.

One class of tensors with an orthogonal basis of eigenvectors are the orthogonally decomposable (odeco) tensors. These have been widely studied; see, e.g., \cite{kolda2001orthogonal, anandkumar2014tensor, robeva2016orthogonal, boralevi2017orthogonal, robeva2017singular}. Symmetric odeco tensors admit a decomposition as in the spectral theorem of a matrix, and the decomposition is essentially unique. 
In this paper, we study symmetric tensors that admit an orthogonal basis of eigenvectors and conditions under which they are unique. As we show, the family of such tensors is much larger than the set of odeco tensors.

Let $S^d(\R^n)\subset \R^{n\times \cdots \times n}$ be the set of real symmetric $n\times \cdots \times n$ tensors of order $d$. Given $\T\in S^d(\R^n)$, we denote its coordinates with respect to the canonical basis of $\R^{n\times \cdots \times n}$ by~$\T_{i_1,\ldots, i_d}$ for $i_1,\ldots,i_d\in [n]$. A tensor $\T$ is symmetric if
$\T_{i_1,\ldots, i_d}=\T_{i_{\sigma(1)},\ldots ,i_{\sigma(d)}}$ for any indices~$1\leq i_1,\ldots,i_d\leq n$ and any permutation $\sigma$ of the set $\{1,\ldots,d\}$.
Given $\T \in S^d(\R^n)$, a unit norm vector $v \in \R^n$ is an \emph{eigenvector} of $\T$ with \emph{eigenvalue} $\lambda \in \R$ if
$\T(\cdot, v, \dots, v) = \lambda v$.
In \cite{qi2005eigenvalues} these are called Z-eigenvectors and Z-eigenvalues, respectively. If $d=2$, this is the standard definition of eigenvectors from linear algebra.

Our first main result generalizes the uniqueness of an orthogonal basis of eigenvectors from symmetric matrices to symmetric tensors. This will settle the conjecture in \cite{mesters2022non}.

\begin{theorem}\label{thm:main-symmetric}
    If $\T\in S^d(\R^n)$ is a generic symmetric tensor with an orthogonal basis of eigenvectors, then this basis is unique (up to sign flip).
\end{theorem}

The analog of the spectral theorem for non-symmetric matrices is the singular value decomposition (SVD): a matrix $M \in \R^{n_1 \times n_2}$ can be decomposed as $M = U \Sigma V^\top$, where $\Sigma \in \R^{n_1 \times n_2}$ is diagonal with diagonal entries $\sigma_i \geq 0$, and $U$ and $V$ are orthogonal matrices; that is, $U^\top U = I_{n_1}$ and $V^\top V = I_{n_2}$, where $I_n$ denotes the $n \times n$ identity matrix. Let $m = \min\{n_1, n_2\}$, then we can write $U \Sigma V^\top = \sum_{i=1}^m \sigma_i u_i \otimes v_i$. The set $\{u_1 \otimes v_1, \dots, u_m \otimes v_m \}$ is a set of orthogonal singular vector pairs of $M$.

Generalizing the SVD from matrices to higher-order tensors is a central topic in multilinear algebra. There are several approaches \cite{lathauwerSVD, draisma2018best, ribot2024decomposing}, but no generalization can inherit all the properties of the SVD \cite{vannieuwenhoven2014generic}. In this paper, we focus on the following property: a generic matrix $M \in \R^{n_1 \times n_2}$ has a unique basis of orthogonal singular vector pairs (up to sign flip). Here generic means that all its singular values $\sigma_i$ are distinct. The following result generalizes this to tensors.

\begin{theorem}\label{thm:main}
    If $\T\in \R^{n_1 \times \cdots \times n_d}$, $(n_1, \dots, n_d) \neq (2,2,2,2)$, is a generic tensor with an orthogonal basis of singular vector tuples, then this basis is unique (up to sign flip).
    If $\T \in \R^{2 \times 2 \times 2 \times 2}$ is a generic tensor with an orthogonal basis of singular vector tuples, then it has four such bases (up to sign flip). 
\end{theorem}

This paper is organized as follows. We recap basics of tensors in \Cref{sec:notation}. We describe the set of tensors with an orthogonal basis of singular vector tuples as the orbit of a linear space under the action of a product of orthogonal groups in \Cref{sec:orthogonal-singular-vectors}. We specialize to symmetric tensors with an orthogonal basis of eigenvectors in \Cref{sec:orthogonal-eigenvectors}. We prove \Cref{thm:main} in \Cref{sec:proof-uniqueness}. We study binary tensors first and then extend to tensors of general format. 
We prove \Cref{thm:main-symmetric} in \Cref{sec:proof-uniqueness-symmetric}.
We study the varieties of tensors (resp. symmetric tensors) with an orthogonal basis of singular vectors (resp. eigenvectors) in \Cref{sec:varieties}.
We propose a structured Tucker decomposition for tensors with an orthogonal basis of singular vectors or eigenvectors in \Cref{sec:multilinear-svd}. We include numerical experiments in \Cref{sec:numerical}.

\section{Background and notation}\label{sec:notation}

In this section, we briefly introduce standard notation and basics of tensors. 

\noindent\textbf{Tensors.} Let $d \geq 2$ and consider tensors in $\R^{n_1 \times \cdots \times n_d}\cong \R^{n_1} \otimes \cdots \otimes \R^{n_d}$ with $n_k \geq 2$ for all~$k$. Without loss of generality, we assume throughout the paper that the factors have been reordered so that $n_1 \leq \cdots \leq n_d$. We regard tensors as multilinear maps and multidimensional arrays. 
For a positive integer $n$, let $[n]:=\{1,\ldots,n\}$ and let $\{e_1, \dots, e_n\}$ denote the canonical basis of $\R^n$. The inner product of two vectors $u,v \in \R^n$ is $\< u, v\> = \sum_{i=1}^n u_i v_i$, which lets us identify $\R^n$ with its dual $(\R^n)^*$. Hence, $\{e_{i_1} \otimes \cdots \otimes e_{i_d} \mid i_k \in [n_k]\}$ is the canonical basis of~$\R^{n_1} \otimes \cdots \otimes \R^{n_d}$. Denote the coordinates of a tensor $\T \in \R^{n_1 \times \cdots \times n_d}$ in this basis by $\T_{i_1, \dots, i_d} = \T(e_{i_1}, \dots, e_{i_d})$.
A tensor has rank one if it is of the form $x^{(1)} \otimes \cdots \otimes x^{(d)}$ for $x^{(k)} \in \R^{n_k} \setminus \{ 0\}$.
The inner product structure on each $\R^{n_k}$ induces an inner product on $\R^{n_1 \times \cdots \times n_d}$ defined on two rank-one tensors as $\<x^{(1)} \otimes \cdots \otimes x^{(d)}, y^{(1)} \otimes \cdots \otimes y^{(d)}\> = \prod_{k \in [d]} \<x^{(k)}, y^{(k)} \>$ and extended to the whole space by bilinearity.
With this inner product, $\{e_{i_1} \otimes \cdots \otimes e_{i_d} \mid i_k \in [n_k]\}$ is an orthonormal basis of $\R^{n_1 \times \cdots \times n_d}$, so $\T_{i_1, \dots, i_d} = \<\T, e_{i_1} \otimes \cdots \otimes e_{i_d} \>$. Hence, the inner product between two tensors $\cS,\T \in \R^{n_1 \times \cdots \times n_d}$ is $\<\cS, \cT \> = \sum_{i_1 = 1}^{n_1} \cdots \sum_{i_d = 1}^{n_d} \cS_{i_1, \ldots, i_d} \T_{i_1, \ldots, i_d}$,
with the induced (Frobenius) norm $\|\cT\|:=\sqrt{\<\cT,\cT\>}$.
Denote $(\R^n)^{\otimes d} = \R^n \otimes \cdots \otimes \R^n \; (d \text{ times})$ and let $v^{\otimes d} = v \otimes \cdots \otimes v \in (\R^n)^{\otimes d}$ for $v \in \R^n$. We can embed $(\R^{n})^{\otimes d}$ in $\R^{n_1 \times \cdots \times n_d}$ provided that $n \leq n_1 = \min_k n_k$. We use the embedding given by considering the basis $\{ e_{i_1} \otimes \cdots \otimes e_{i_d} \mid i_k \in [n]\}$ of $(\R^n)^{\otimes d}$ as a subset of the basis $\{ e_{i_1} \otimes \cdots \otimes e_{i_d} \mid i_k \in [n_k]\}$ of $\R^{n_1} \otimes \cdots \otimes \R^{n_d}$. Therefore, consider $v^{\otimes d}$ with $v \in \R^{n} \setminus \{ 0\}$ as a rank-one tensor in $\R^{n_1 \times \cdots \times n_d}$.

\medskip

\noindent\textbf{Change of basis.}
For a tuple of matrices $\Q = (Q_1, \dots, Q_d)$ with $Q_k = (q^{(k)}_{i_kj_k}) \in \R^{m_k \times n_k}$ and a tensor $\T \in \R^{n_1 \times \cdots \times n_d}$ let $\Q \cdot \T = (Q_1,\ldots,Q_d)\cdot \T\in \R^{m_1\times \cdots \times m_d}$ be defined as
\[
[(Q_1,\ldots,Q_d)\cdot \T]_{i_1,\ldots,i_d}\;=\;\sum_{j_1=1}^{n_1} \cdots\sum_{j_d=1}^{n_d} q^{(1)}_{i_1 j_1}\cdots q^{(d)}_{i_d j_d} \T_{j_1,\dots, j_d}.
\]
For example, when $d=2$ we have $(Q_1, Q_2) \cdot \T = Q_1 \T Q_2^\top$. 
Note that $(I_{n_1}, \dots, I_{n_d}) \cdot \T = \T$, where $I_n \in \R^{n \times n}$ is the identity matrix. Given two tuples $(Q_1, \dots, Q_d)$ and $(\tilde{Q}_1, \dots, \tilde{Q}_d)$ with $Q_k \in \R^{l_k \times m_k}$ and $\tilde{Q}_k \in \R^{m_k \times n_k}$ we have
\[
 (Q_1, \dots, Q_d) \cdot \big( (\tilde{Q}_1, \dots, \tilde{Q}_d) \cdot \T \big) = (Q_1\tilde{Q}_1, \dots, Q_d\tilde{Q}_d) \cdot \T.
\]
Given a positive integer $n$, $\rO(n) = \{Q \in \R^{n \times n} \mid Q^\top Q = I_n\}$ is the orthogonal group in dimension $n$.
Tuples $\Q = (Q_1, \dots, Q_d) \in \rO(n_1) \times \cdots \times \rO(n_d)$ give a group action of $\rO(n_1) \times \cdots \times \rO(n_d)$ on $\R^{n_1 \times \cdots \times n_d}$. The group action is linear: $\Q \cdot (\lambda \T + \lambda' \T') = \lambda \Q \cdot \T + \lambda' \Q \cdot \T'$, where $\lambda,\lambda' \in \R$. For a rank-one tensor $x = x^{(1)} \otimes \cdots \otimes x^{(d)}$, we have $x_{i_1, \dots, i_d} = x^{(1)}_{i_1} \cdots x^{(d)}_{i_d}$, so $\Q \cdot x = (Q_1 x^{(1)}) \otimes \cdots \otimes (Q_dx^{(d)})$. Hence, the action of a fixed $\Q \in \rO(n_1) \times \cdots \times \rO(n_d)$ on $\R^{n_1 \times \cdots \times n_d}$ corresponds to the orthogonal change of basis from $\{e_{i_1} \otimes \cdots \otimes e_{i_d} \mid i_k \in [n_k]\}$ to $\{q^{(1)}_{i_1} \otimes \cdots \otimes q^{(d)}_{i_d} \mid i_k \in [n_k]\}$, where $q^{(k)}_j$ is the $j$-th column of $Q_k$ for $k \in [d]$ and $j \in [n_k]$.

\section{Tensors with orthogonal singular vectors}\label{sec:orthogonal-singular-vectors} 
The goal of this section is to characterize the tensors in $\R^{n_1\times \cdots \times n_d}$ with an orthogonal basis of singular vector tuples.

\begin{definition}\label{def:SVT}
Given $\T\in \R^{n_1 \times \cdots \times n_d}$, a tuple of vectors $(x^{(1)}, \ldots,  x^{(d)}) \in \R^{n_1} \times \cdots \times \R^{n_d}$ with $\| x^{(k)}\| = 1$ for all $k \in [d]$ is a \emph{singular vector tuple} of $\T$ with \emph{singular value} $\lambda \in \R$ if
\[
\T(x^{(1)}, \dots, x^{(k-1)}, \cdot, x^{(k+1)}, \dots, x^{(d)}) = \lambda x^{(k)} \quad \text{for all } k \in [d].
\]
Two singular vector tuples $x_1=(x^{(1)}_1, \ldots,  x^{(d)}_1)$ and $x_2=(x_2^{(1)}, \ldots,  x_2^{(d)})$ are \emph{orthogonal} if $\langle x_1^{(k)}, x_2^{(k)} \rangle = 0$ for all $k \in [d]$. 
A set of singular vector tuples $\{x_1,\ldots,x_{n_1} \}$ is an \emph{orthogonal basis of singular vector tuples} if $x_i$ and $x_j$ are orthogonal for all $i \neq j$.
\end{definition}

\begin{remark}
    We identify a singular vector tuple $x=(x^{(1)}, \dots, x^{(d)})$ with the rank-one tensor $x^{(1)} \otimes \cdots \otimes x^{(d)}$. Our definition for two singular vector tuples to be orthogonal is more restrictive than them being orthogonal as rank-one tensors: they must be orthogonal on all factors. This generalizes the orthogonality structure in the SVD for matrices.
\end{remark}

The singular vector tuples are the critical points of the function $f(\lambda,x^{(1)},\ldots,x^{(k)})=\|\cT-\lambda x^{(1)}\otimes \cdots\otimes x^{(d)}\|$, where $\lambda\in \R$ and $\|x^{(k)}\|=1$ for all $k$. A minimizer of this function is a best rank-one approximation of $\cT$; see \cite{zhang2001rank,lim2005singular,ribot2024decomposing}.

\begin{proposition}\label{prop:X}
The set of tensors in $\R^{n_1 \times \cdots \times n_d}$ with an orthogonal basis of singular vector tuples is
\[
X \coloneqq (\rO(n_1) \times \cdots \times \rO(n_d) ) \cdot V=\{ \Q \cdot \T \mid \Q\in \rO(n_1) \times \cdots \times \rO(n_d),\,\T\in V\},
\]
where
\[
V  \coloneqq \{ \T \in \R^{n_1 \times \cdots \times n_d}  \mid \T_{i_1,j,j,\dots,j} =\cdots = \T_{j,j,\dots,j,i_d} = 0 \text{ for all } j \in [n_1], i_k \in [n_k]\setminus j\}.
\]
That is, the set $X$ consists of tensors that lie in $V$ after an orthogonal change of basis.
\end{proposition}
Before proving this result, we provide some discussion and relevant lemmas.
\begin{example}\label{ex:222}
    For $d=2$, $V$ is the set of diagonal matrices and $X$ equals $\R^{n_1 \times n_2}$, by the SVD. For $2 \times 2 \times 2$ tensors, 
    $V = \{ \cT \in \R^{2\times 2\times 2} \mid \cT_{112} = \cT_{121}  = \cT_{211} = \cT_{122} = \cT_{212} = \cT_{221} = 0\}$
    is the set of diagonal tensors, and $X$ is the set of \emph{odeco} tensors.
\end{example}

The subspace $V$ defined in \Cref{prop:X} plays a central role. It consists of tensors such that the coordinates that differ from a diagonal coordinate in one entry are zero. Formally,
\begin{equation*}
    V = \{ \T \in \R^{n_1 \times \cdots \times n_d} \mid \T_{i_1, \dots, i_d} = 0 \text{ if } \exists\, j \in [n_1] \text{ s.t. } d_H((i_1, \dots, i_d), (j, \dots, j)) = 1 \},
\end{equation*}
where $d_H$ is the Hamming distance on $[n_1] \times \cdots \times [n_d]$ given by
$
d_H(\bi, \bj) = |\{ k \mid i_k \neq j_k\}|
$. Recall that we assume throughout the paper that $n_1\leq \cdots \leq n_d$. The following result shows the relevance of $V$.

\begin{lemma}\label{lem:Vcanonical}
    Let $\T \in \R^{n_1 \times \cdots \times n_d}$. Then, $e_1^{\otimes d},\ldots,e_{n_1}^{\otimes d}$ are singular vector tuples of $\T$ if and only if $\T \in V$.
\end{lemma}
\begin{proof}
    Given $j \in [n_1]$, $e_j^{\otimes d}$ is a singular vector tuple of $\T$ if and only if contracting $\T$ with $e_j$ in all but one factors gives a vector that is parallel to $e_j$.
    Equivalently, $\T(e_j, \dots, e_j, e_{i_k}, e_j, \dots, e_j) = \T_{j, \dots, j, i_k, j \dots, j} = 0$
    for all $k \in [d]$ and all $i_k \in [n_k] \setminus \{j\}$.
    This holds for all $j \in [n_1]$ if and only if $\T \in V$.
\end{proof}

The following result shows that when we act on a tensor with a tuple of orthogonal matrices, the singular vector tuples transform in an equivariant way.

\begin{lemma}[Equivariance]\label{lem:equivariance-singular-vectors} Given $\Q \in\rO(n_1) \times \cdots \times  \rO(n_d)$, the tuple $(x^{(1)} , \dots, x^{(d)})$ is a singular vector tuple of $\T\in \R^{n_1 \times \cdots \times n_d}$ if and only if $(Q_1x^{(1)}, \dots, Q_d x^{(d)})$ is a singular vector tuple of $\Q \cdot \T$. The singular values of $\T$ are equal to the singular values of $ \Q \cdot \T$.
\end{lemma}

\begin{proof}
    The tuple $(x^{(1)}, \dots, x^{(d)})$ is a singular vector tuple of $\T$ if and only if there exists~$\lambda \in \R$ such that for all $k \in [d]$
    \[
    \left((x^{(1)})^\top, \dots, (x^{(k-1)})^\top, I_{n_k}, (x^{(k+1)})^\top, \ldots, (x^{(d)})^\top \right) \cdot \T = \lambda x^{(k)}.
    \]
    Then for all $k \in [d]$ we have
    \begin{align*}
        &\left((Q_1 x^{(1)})^\top, \dots, (Q_{k-1} x^{(k-1)})^\top, I_{n_k}, (Q_{k+1} x^{(k+1)})^\top, \ldots, (Q_d x^{(d)})^\top \right) \cdot \big( (Q_1, \dots, Q_d) \cdot \T \big)\\
        &=  \left((x^{(1)})^\top, \dots, (x^{(k-1)})^\top, Q_k, (x^{(k+1)})^\top, \ldots, (x^{(d)})^\top \right) \cdot \T \\
        & = 
        Q_k \left(
        \left((x^{(1)})^\top, \dots, (x^{(k-1)})^\top, I_{n_k}, (x^{(k+1)})^\top, \ldots, (x^{(d)})^\top \right) \cdot \T  \right) \\
        &= \lambda Q_k x^{(k)}.
    \end{align*}
    Hence $(Q_1x^{(1)}, \dots, Q_d x^{(d)})$ is a singular vector tuple of $(Q_1 \dots, Q_d) \cdot \T$. The converse implication is analogous and follows by using $Q_{k}^\top Q_k = I_{n_k}$ for all $k \in [d]$. The singular values of $\T$ and $\Q \cdot \T$ coincide by orthogonality of $Q_k$.
\end{proof}

\begin{proof}[Proof of \Cref{prop:X}]
Fix $\cS \in \R^{n_1 \times \cdots \times n_d}$ with an orthogonal basis of singular vector tuples $q_1, \dots, q_{n_1}$, where $q_j = (q_j^{(1)}, \dots, q_j^{(d)})$. For each $k \in [d]$, fix $Q_k \in \rO(n_k)$ such that the first $n_1$ columns of $Q_k$ are $q_1^{(k)}, \dots, q_{n_1}^{(k)}$. Let $\T = (Q_1^\top, \dots, Q_d^\top) \cdot \cS$. Then, $e_1^{\otimes d},\dots, e_{n_1}^{\otimes d}$ are singular vector tuples of $\T$ by \Cref{lem:equivariance-singular-vectors}, so $\T \in V$ by \Cref{lem:Vcanonical}. Therefore, $\cS = (Q_1, \dots, Q_d) \cdot \T \in X$. Similarly, every tensor in $X$ has an orthogonal basis of singular vector tuples, by Lemmas \ref{lem:Vcanonical} and \ref{lem:equivariance-singular-vectors}. 
\end{proof}

By comparison, a tensor $\T \in \R^{n_1 \times \cdots \times n_d}$ is \emph{orthogonally decomposable} (\emph{odeco}, for short) if it has a decomposition
\[
\T = \sum_{j = 1}^{n_1} \lambda_j q_j^{(1)} \otimes \cdots \otimes q_j^{(d)},
\]
where $\lambda_j \in \R$ and $\langle q^{(k)}_i, q^{(k)}_j \rangle = 0$ for all $k \in [d]$ and $i \neq j$.
The set of odeco tensors is~$X_{\mathrm{odeco}} \coloneqq (\rO(n_1) \times \cdots \times \rO(n_d)) \cdot V_{\diag}$, where
$V_{\diag}$  is the set of diagonal tensors
\[
V_{\diag} \coloneqq \{ \T \in \R^{n_1 \times \cdots \times n_d} \mid \T_{i_1, \cdots, i_d} \neq 0 \text{ only if } i_1 = \cdots = i_d\} .
\]
Since $V_{\diag} \subseteq V$, then $X_{\mathrm{odeco}} \subseteq X$: each summand in an odeco decomposition is a singular vector tuple of the tensor, so odeco tensors have an orthogonal basis of singular vector tuples.
Based on \Cref{ex:222}, $V = V_{\diag}$ for $2\times 2 \times 2$ tensors. But this case is special: for all other formats $n_1 \times \cdots \times n_d$ with $d \geq 3$, we have strict containment  $V_{\diag} \subsetneq V$. See \Cref{fig:sparsity-patterns}, where boxes are nonzero entries and diagonal entries are highlighted for visualization. 

\begin{figure}[ht]
    \centering
    \begin{subfigure}{0.3\linewidth}
        \centering
    \includegraphics[width=0.5\linewidth]{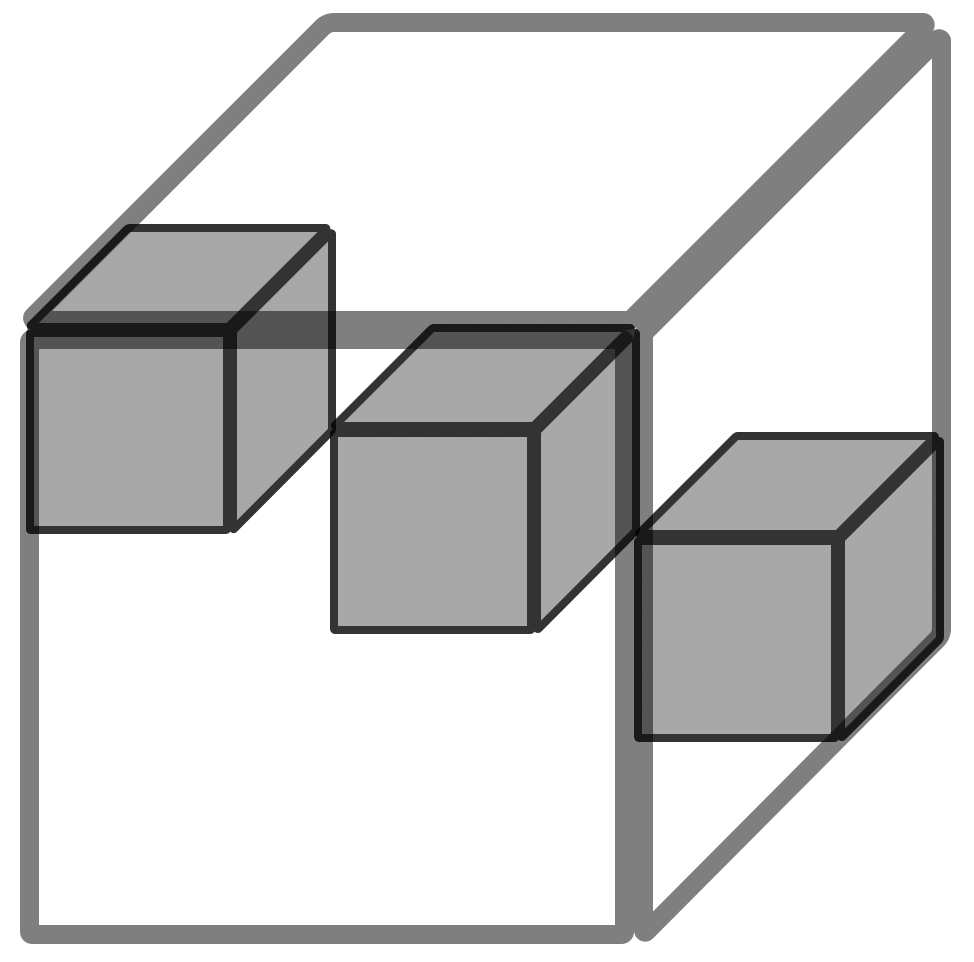}
    \caption{Tensors in $V_{\diag}$}
    \label{fig:sparsity_Vdiag}
    \end{subfigure}
    \begin{subfigure}{0.3\linewidth}
        \centering
    \includegraphics[width=0.5\linewidth]{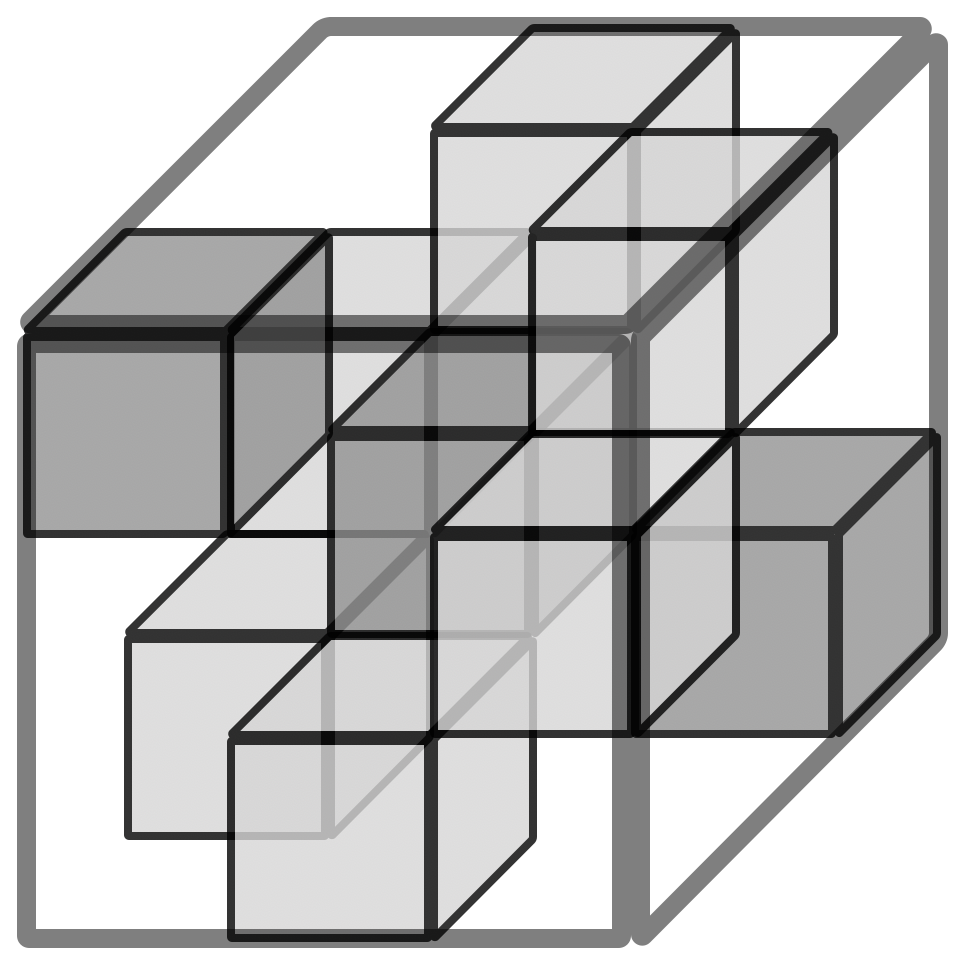}
    \caption{Tensors in $V$}
    \label{fig:sparsity-V}
    \end{subfigure}
    \caption{Tensors in $\R^{3 \times 3 \times 3}$ with different sparsity patterns}
    \label{fig:sparsity-patterns}
\end{figure}

It is well-known that odeco decompositions are unique (up to reordering of the summands).
That is, if $\T \in X_\mathrm{odeco}$ is generic, then it has a decomposition as $\T = \Q \cdot \cS$ with $\Q \in \rO(n_1) \times \cdots \times \rO(n_d)$ and $\cS \in V_{\diag}$ that is unique in the following sense:
$\cS$ and the first $n_1$ columns of $Q_1, \dots, Q_d$ are unique (up to sign and permutation). \Cref{thm:main} generalizes this result to a larger set of tensors, as follows. If $\T \in X$ is generic, then it has a decomposition $\T = \Q \cdot \cS$ with $\Q \in \rO(n_1) \times \cdots \times \rO(n_d)$ and $\cS \in V$ that is unique in the following sense: the leading subtensor of shape $n_1 \times \cdots \times n_1$ of $\cS$ and the first $n_1$ columns of $Q_1, \dots, Q_d$ are unique (up to sign and permutation). See \Cref{sec:multilinear-svd} for details.

\section{Symmetric tensors with orthogonal eigenvectors}\label{sec:orthogonal-eigenvectors}

A tensor $\T \in (\R^n)^{\otimes d}$ is \emph{symmetric} if $\T_{i_1, \dots, i_d} = \T_{i_{\sigma(1)}, \dots, i_{\sigma(d)}}$ for any $1 \leq i_1,\dots,i_d \leq n$ and any permutation $\sigma \in S_d$. In this section, we specialize Section~\ref{sec:orthogonal-singular-vectors} to symmetric tensors. 

Let $S^d(\R^n) \subset (\R^{n})^{\otimes d}$ denote the space of real symmetric $n \times \cdots \times n$ tensors of order $d$. The orthogonal group $\rO(n)$ acts on $S^d(\R^n)$ as follows.
Given $Q \in \rO(n)$ and $\T \in S^d(\R^n)$, we have $Q \bullet \T \coloneqq (Q, \dots, Q) \cdot \T \in S^d(\R^n)$. For symmetric tensors, we replace the notion of singular vector tuples by that of eigenvectors.

\begin{definition} \label{def:eigenvector}
    Given $\T \in S^d(\R^n)$, a unit norm vector $v \in \R^n$ is an \emph{eigenvector} of $\T$ if $v^{\otimes d}$ is a singular vector tuple of $\T$. That is, there exists an \emph{eigenvalue} $\lambda \in \R$ such that
    \[
    \T (\cdot, v, \dots, v) = \lambda v.
    \]
    An \emph{orthogonal basis of eigenvectors} is a set of~$n$ eigenvectors orthogonal to each other.
\end{definition}

The analog of \Cref{prop:X} for symmetric tensors is the following.

\begin{proposition}\label{prop:X-symmetric}
    The set of all symmetric tensors in $S^d(\R^n)$ with an orthogonal basis of eigenvectors is
    \[
    X_\sym \coloneqq \rO(n) \bullet V_\sym = \{ Q \bullet \T \mid Q \in \rO(n), \T \in V_\sym \},
    \]
    where
    \[
    V_\sym \coloneqq \{ \T \in S^d(\R^n) \mid \T_{i,j,\dots,j} = 0 \text{ for all } i \neq j\} = V \cap S^d(\R^n).
    \]
\end{proposition}

\begin{example}
    For $d=2$, $V_\sym$ is the set of diagonal matrices and $X_\sym $ equals $ S^2(\R^n)$, by the spectral theorem.
\end{example}

\Cref{prop:X-symmetric} follows from adapting the proof of \Cref{prop:X} to the symmetric case. Lemmas~\ref{lem:Vcanonical} and \ref{lem:equivariance-singular-vectors} can be specialized to symmetric tensors, and their proofs are analogous. We include the statements for completeness. 

\begin{lemma} \label{lem:Vcanonical-symmetric}
    Let $\T \in S^d(\R^n)$. Then, $e_1, \dots, e_n$ are eigenvectors of $\T$ if and only if $\T \in V_\sym$.
\end{lemma}

\begin{lemma} \label{lem:equivariance-eigenvectors}
    Given $Q \in \rO(n)$, $v \in \R^n$ is an eigenvector of $\T \in S^d(\R^n)$ if and only if $Qv$ is an eigenvector of $Q \bullet \T$. The eigenvalues of $\T$ are equal to the eigenvalues of $Q\bullet \T$.
\end{lemma}
Symmetric odeco tensors have an orthogonal basis of eigenvectors: the set of symmetric odeco tensors is $X_{\sym - \mathrm{odeco}} \coloneqq \rO(n) \bullet V_{\diag}$ and is contained in $ X_{\sym}$, because $V_{\diag} \subseteq V_{\sym}$.

\section{Uniqueness of orthogonal singular vectors} \label{sec:proof-uniqueness}
In this section, we prove \Cref{thm:main}.
The main idea is to characterize tensors with two different orthogonal bases of singular vector tuples and to show that such tensors are non-generic in $X$; i.e., they are contained in a proper algebraic subvariety. We study binary tensors first and then reduce the analysis of the general case to the binary case using standard properties of normal operators.

Denote the set of signed permutation matrices by ${\rm SP}(n)$; that is, $Q=PD \in \rO(n)$ with~$P$ a permutation matrix and $D$ diagonal with diagonal entries $\pm 1$. If $Q\in {\rm SP}(n)$ then the columns of $Q$ give the canonical basis of~$\R^n$ (up to sign). We reformulate \Cref{thm:main}.

\begin{proposition}\label{prop:main-equivalence}
    \Cref{thm:main} for $(n_1, \ldots, n_d) \neq (2,2,2,2)$ is equivalent to the following statement: Consider a generic $\cS \in V \subset \R^{n_1 \times \cdots \times n_d}$ and let $\Q \in \rO(n_1) \times \cdots \times \rO(n_d)$. Then, $\Q \cdot \cS \in V$ if and only if
    \begin{equation}\label{eq:Qspec}
        Q_1 \in \rSP(n_1) \quad \text{ and } \quad Q_k = \begin{pmatrix}
        Q_1D_k & 0 \\
        0 & \tilde{Q}_k
    \end{pmatrix} \quad \text{for all } k \geq 2,
    \end{equation}
    where $D_k$ is a diagonal matrix with diagonal entries $\pm 1$ and $\tilde{Q}_k \in \rO(n_k - n_1)$.
\end{proposition}
\begin{proof}
Given $\Q = (Q_1, \dots, Q_d) \in \rO(n_1) \times \cdots \times \rO(n_d)$, let $q_j^{(k)}$ denote the $j$-th column of $Q_k$ for $k \in [d]$ and $j \in [n_k]$. Then, \eqref{eq:Qspec} is equivalent to
\begin{equation}\label{eq:Qspec2}
    \{ q_j^{(1)} \otimes \cdots \otimes q_j^{(d)} \mid j \in [n_1] \} \;=\; \{ \pm e_j^{\otimes d} \mid j \in [n_1] \}.    
\end{equation}
Uniqueness of the orthogonal basis of singular vector tuples holds for generic tensors in $V$ if and only if it holds for generic tensors in $X$, since $\cS\in V$ has a unique orthogonal basis if and only if $\Q\cdot \cS$ has a unique orthogonal basis, by Lemma~\ref{lem:equivariance-singular-vectors}.

Let (S1) be the statement of Theorem~\ref{thm:main} and let (S2) be the statement in this proposition. We first show that (S2) implies (S1). Consider a generic $\cT\in X$ and assume (S2) holds.
Suppose that $\T$ has two distinct orthogonal bases of singular vector tuples: $\cT=\Q\cdot \cS=\tilde \Q\cdot \tilde \cS$ for some $\cS,\tilde \cS\in V$ and $\Q,\tilde\Q \in \rO(n_1) \times \cdots \times \rO(n_d)$ such that the first $n_1$ columns of $Q_k$ and $\tilde{Q}_k$ not equal (up to sign flip) for some $k \in [d]$. Since $\tilde{\cS} = \tilde{\Q}^\top \cdot (\Q \cdot \cS)$, we conclude that the columns of $\tilde{Q}_k^\top Q_k$ satisfy \eqref{eq:Qspec2} for all $k \in [d]$, by (S2). Equivalently, the first $n_1$ columns of $Q_k$ and $\tilde{Q}_k$ are equal (up to sign) for all~$k \in [d]$, a contradiction, so (S2) implies (S1).

Suppose (S1) holds. Since $V\subset X$, we can use (S1) to conclude that if $\cS \in V$ is generic in $V$ (and so generic in $X$) then it has a unique orthogonal basis of singular vectors. Since $\cS\in V$ the basis is $\{\pm e_j^{\otimes d}\mid j \in [n_1]\}$. Thus, if $\Q\cdot \cS\in V$ for some~$\Q$, then~$\Q$ satisfies \eqref{eq:Qspec2}.
\end{proof}

\subsection{Binary tensors}

In this section, we prove \Cref{thm:main} for binary tensors; that is, for tensors $\cT \in (\R^2)^{\otimes d}$. 
Denote the canonical basis of $\R^2$ by $\{e_0, e_1\}$.
We prove the case~$\R^{2 \times 2 \times 2}$ by computing the singular vector tuples directly.

\begin{proposition} \label{prop:main-binary-d3}
    \Cref{thm:main} holds for tensors in $\R^{2 \times 2 \times 2}$.
\end{proposition}
\begin{proof}
For $2\times 2\times 2$ tensors the subspace $V$ consists of diagonal tensors. The orbit $X = \rO(2)^{3} \cdot V$ is the odeco tensors. Generic odeco tensors have a unique basis of orthogonal singular vector tuples, as follows: $u \otimes v \otimes w$ is a singular vector tuple of $\T =  \lambda_0 e_0^{\otimes 3} + \lambda_1 e_1^{\otimes 3}$~if
\[
\det \begin{pmatrix}
        \lambda_0 v_0w_0 & u_0 \\
        \lambda_1 v_1 w_1 & u_1 
    \end{pmatrix} = \det  \begin{pmatrix}
        \lambda_0u_0w_0 & v_0 \\ 
        \lambda_1 u_1w_1 & v_1
    \end{pmatrix} = \det  \begin{pmatrix}
        \lambda_0 v_0u_0 & w_0 \\ 
        \lambda_1 u_1v_1 & w_1
    \end{pmatrix} = 0.
\]
Hence, for nonzero $\lambda_0, \lambda_1 \in \R$, the tensor $\T$ has six singular vector tuples (up to scaling):
\[
\arraycolsep=0.3pt\def\arraystretch{0.8}
\begin{array}{ccc}
    \begin{pmatrix}
        1 \\ 0
    \end{pmatrix}
    \otimes 
    \begin{pmatrix}
        1 \\ 0
    \end{pmatrix}
    \otimes 
    \begin{pmatrix}
        1 \\ 0
    \end{pmatrix}, 
    & 
    \begin{pmatrix}
        0 \\ 1
    \end{pmatrix}
    \otimes 
    \begin{pmatrix}
        0 \\ 1
    \end{pmatrix}
    \otimes 
    \begin{pmatrix}
        0 \\ 1
    \end{pmatrix},
    &
    \begin{pmatrix}
        \lambda_1 \\ \lambda_0
    \end{pmatrix}
    \otimes 
    \begin{pmatrix}
        \lambda_1 \\ \lambda_0
    \end{pmatrix}
    \otimes 
    \begin{pmatrix}
        \lambda_1 \\ \lambda_0
    \end{pmatrix},
    \\[1em]
    \begin{pmatrix}
        \lambda_1 \\ \lambda_0
    \end{pmatrix}
    \otimes 
    \begin{pmatrix}
        \lambda_1 \\ -\lambda_0
    \end{pmatrix}
    \otimes 
    \begin{pmatrix}
        \lambda_1 \\ -\lambda_0
    \end{pmatrix},
    & 
    \begin{pmatrix}
        \lambda_1 \\ -\lambda_0
    \end{pmatrix}
    \otimes 
    \begin{pmatrix}
        \lambda_1 \\ \lambda_0
    \end{pmatrix}
    \otimes 
    \begin{pmatrix}
        \lambda_1 \\ -\lambda_0
    \end{pmatrix},
    &
    \begin{pmatrix}
        \lambda_1 \\ -\lambda_0
    \end{pmatrix}
    \otimes 
    \begin{pmatrix}
        \lambda_1 \\ -\lambda_0
    \end{pmatrix}
    \otimes 
    \begin{pmatrix}
        \lambda_1 \\ \lambda_0
    \end{pmatrix}.
\end{array} 
\]
The only pair orthogonal to each other are $e_0^{\otimes 3}$ and $e_1^{\otimes 3}$.
\end{proof}

We now assume $d \geq 4$. First, we give a more precise statement of the uniqueness of the orthogonal singular vectors for binary tensors of order $4$. This explains why this curious case is ruled out in \Cref{thm:main}.
The proof appears later in this section.

\begin{theorem} \label{thm:main-2222}
    For $ \T \in \R^{2 \times 2 \times 2 \times 2}$, if $\{x_0^{(1)} \otimes x_0^{(2)} \otimes x_0^{(3)} \otimes x_0^{(4)} , x_1^{(1)} \otimes x_1^{(2)} \otimes x_1^{(3)} \otimes x_1^{(4)}\}$ is an orthogonal basis of singular vector tuples of $\T$, then so are
    \begin{align*}
        &\{x_0^{(1)} \otimes x_0^{(2)} \otimes x_1^{(3)} \otimes x_1^{(4)} , x_1^{(1)} \otimes x_1^{(2)} \otimes x_0^{(3)} \otimes x_0^{(4)}\}, \\
        &\{x_0^{(1)} \otimes x_1^{(2)} \otimes x_0^{(3)} \otimes x_1^{(4)} , x_1^{(1)} \otimes x_0^{(2)} \otimes x_1^{(3)} \otimes x_0^{(4)}\}, \\
        &\{x_0^{(1)} \otimes x_1^{(2)} \otimes x_1^{(3)} \otimes x_0^{(4)} , x_1^{(1)} \otimes x_0^{(2)} \otimes x_0^{(3)} \otimes x_1^{(4)}\}.
    \end{align*}
    Generically, these are the only orthogonal bases of singular vector tuples (up to sign flip).
\end{theorem}
\noindent We postpone the proof of this result until the end of this subsection.

The coordinates of a binary tensor $\T$ are indexed by binary strings of length $d$. Given $\bi = (i_1, \dots, i_d) \in \{0,1\}^{d}$, let $|\bi| = \sum_{k} i_k$.  
Let $\cI \coloneqq \{ \bi \in \{ 0,1\}^d \mid |\bi| \in \{1,d-1\}\}$.
Then 
\[
V = \{ \cT \in (\R^2)^{\otimes d} \mid \cT_{\bi} = 0 \text{ if } \bi \in \cI \}.
\]
Given a fixed tuple $\Q \in \rO(2)^{d}$, we consider the linear space $\Q \cdot V \coloneqq \{ \Q \cdot \T \mid \T \in V\}$. 
We characterize when $\T \in V$ and $\Q \cdot \T \in V$, motivated by \Cref{prop:main-equivalence}.

The set $V$ has dimension $2^d - 2d$ in $(\R^2)^{\otimes d}$. 
We analyze its orthogonal complement \[
V^\perp = \{ \cT \in (\R^2)^{\otimes d} \mid \cT_{\bi} = 0 \text{ if } \bi \notin \cI \},
\]
which is smaller since
$2^{d}-2d \geq 2d$ for $d \geq 4$. 
The map $\T\mapsto \Q\cdot \T$ has coordinates
 \[
    [\Q \cdot \T]_\bi = \sum_{\bj \in \{ 0, 1\}^d } q^{(1)}_{i_1 j_1} \cdots q^{(d)}_{i_d j_d} \T_{\bj},
    \]
    where $\bi = (i_1, \ldots, i_d)$  and $\bj = (j_1, \ldots, j_d)$.
Vectorizing $\T$, the map can be represented by a $2^d\times 2^d$ matrix with $(\bi, \bj)$ entry 
$q^{(1)}_{i_1 j_1} \cdots q^{(d)}_{i_d j_d}$.
We define the $(2^d - 2d) \times 2d$ submatrix $M_\Q$ with 
rows indexed by $\{0,1\}^{d} \setminus \cI$ and columns indexed by $\cI$:
    \[
    [M_\Q]_{(\bi, \bj)} = q^{(1)}_{i_1 j_1} \cdots q^{(d)}_{i_d j_d}.
    \]

\begin{example}\label{ex:MQd4}
Let $d=4$ and $\Q = (P, I, I, I)$, where $I \coloneqq \begin{pmatrix}
        1 & 0 \\ 0 & 1
    \end{pmatrix} 
    \text{ and } P \coloneqq \begin{pmatrix}
        0 & 1 \\ 1 & 0
    \end{pmatrix}.$ Then,
    \[
    M_{\Q} = 
\begin{blockarray}{ccccccccc}
    & 0001 & 0010 & 0100 & 0111 & 1000 & 1011 & 1101 & 1110 \\
    \begin{block}{c(cccccccc)}
      0000 & 0&0&0&0&1&0&0&0\\
      0011 & 0&0&0&0&0&1&0&0\\
      0101 & 0&0&0&0&0&0&1&0\\
      0110 & 0&0&0&0&0&0&0&1\\
      1001 & 1&0&0&0&0&0&0&0\\
      1010 & 0&1&0&0&0&0&0&0\\
      1100 & 0&0&1&0&0&0&0&0\\
      1111 & 0&0&0&1&0&0&0&0\\
    \end{block}
    \end{blockarray}
    \;.
    \]
In contrast, if $\Q = (I, I, I, I)$ or $\Q = (P, P, P, P)$, then $M_\Q \in \R^{8 \times 8}$ is the zero matrix.
\end{example}

The idea of the proof of \Cref{thm:main} for binary tensors is to study the set of tensors with more than one orthogonal basis of singular vector tuples. To show that a generic tensor in $V$ has a unique basis of singular vectors, we show that the codimention of $(\Q \cdot V) \cap V$ in $V$, denoted $\codim((\Q \cdot V) \cap V, V)$,
is large enough when $\Q$ corresponds to an orthogonal basis of singular vector tuples different from $\{e_0^{\otimes d}, e_1^{\otimes d} \}$.
This codimension can be found from $M_\Q$.

\begin{lemma} \label{lem:MQ}
     Let $\Q = (Q_1, \dots, Q_d)\in \rO(2)^{d}$. Then,
     \[
     \codim((\Q \cdot V) \cap V, V) =  \codim((\Q \cdot V^\perp) \cap V^\perp, V^\perp)=\rank(M_\Q).
     \]
\end{lemma}

\begin{proof}
    First, we show that $ \codim((\Q \cdot V) \cap V, V) =  \codim((\Q \cdot V^\perp)\cap V^\perp, V^\perp)$.  Let $U=\Q \cdot V$. Then $U^\perp= \Q \cdot V^\perp$, as orthogonal $\Q$ preserve the inner product.  Note that~$\dim(U) = \dim(V)$ and $\dim(U^\perp) = \dim(V^\perp)$. Using Grassmann's identity we get
    \begin{align*}
        \codim(U \cap V, V) &= \codim(V^\perp, (U \cap V)^\perp) = \dim((U \cap V)^\perp) - \dim(V^\perp)\\
        &= \dim(V^\perp + U^\perp) - \dim(V^\perp) = \dim(U^\perp) - \dim (U^\perp \cap V^\perp)\\
        &= \dim(V^\perp) - \dim (U^\perp \cap V^\perp) = \codim(U^\perp \cap V^\perp, V^\perp).
    \end{align*}
The kernel of $M_\Q$ says how many new restrictions are imposed on $\T \in V^\perp$ by also imposing~$\Q \cdot \T \in V^\perp$, as follows.
For $\T \in V^\perp$ we have $\T_{\bj} \neq 0$ only if $\bj \in \cI$, so
    \[
    [\Q \cdot \T]_\bi = \sum_{\bj \in \cI} q^{(1)}_{i_1 j_1} \cdots q^{(d)}_{i_d j_d} \T_{j_1 \dots j_d}.
    \]
    Imposing $\Q \cdot \T \in V^\perp$ implies $[\Q \cdot \T]_\bi = 0$ for all $\bi \notin \cI$. This is a system of $2^d - 2d$ linear equations in the variables $\{\T_\bi \mid \bi \in \cI\}$. The matrix corresponding to this linear system is $M_\Q$, so $\Q \cdot \T \in V^\perp$ if and only if $(\T_\bi \mid \bi \in \cI) \in \ker M_\Q$. Hence, 
    $\dim ((\Q \cdot V^\perp) \cap V^\perp) = \dim(\ker M_\Q)$. Finally, $\codim((\Q \cdot V^\perp) \cap V^\perp, V^\perp) = \rank(M_\Q)$ by the rank-nullity theorem.
\end{proof}

We move from a fixed tuple $\Q$ to study the codimension as we vary over all tuples $\Q \in \rO(2)^d$ that define a new orthogonal basis of singular vectors.  First, we assume that~$\Q$
does not consist of signed permutation matrices. 
We show that
\begin{equation}\label{eq:positive-codimension}
\codim \left(\bigcup_{\Q \in \rO(2)^{d}\setminus {\rm SP}(2)^d} (\Q \cdot V) \cap V, V \right) >0.
\end{equation} 
Our arguments rely on the following results that reveal the structure of this problem.
\begin{lemma} \label{lem:codimension-not-in-SP}
    If $\Q = (Q_1, \dots, Q_d) \in \rO(2)^{d}$ with $d \geq 4$ and there are exactly $r > 0$ matrices~$Q_k$ such that $Q_k \notin \rSP(2)$, then $\codim((\Q \cdot V) \cap V, V) \geq r$.
\end{lemma}
\begin{proof}
    We show that $\dim((\Q \cdot V^\perp) \cap V^\perp) \leq 2d - r$, which is equivalent, by \Cref{lem:MQ}. First, suppose that $Q_k \notin \rSP(2)$ for some $k \in [d]$. Let $\T \in V^\perp \cap (\Q \cdot V^\perp)$. We claim that 
    \[
    f_k(\T) := \sum_{\bi \in \cI \mid i_k = 0} \T_\bi^2 - \sum_{\bi \in \cI \mid i_k = 1} \T_\bi^2 = 0.
    \]
    The $k$-th principal flattening of $\T \in (\R^2)^{\otimes d}$ is the matrix $\T^{(k)} \in \R^{2 \times 2^{d-1}}$ with rows indexed by $i_k \in \{0,1\}$ and columns by tuples in $\{0,1\}^{d-1}$ with 
$\T^{(k)}_{i_k, (i_1, \dots, i_{k-1}, i_{k+1}, \dots, i_d)} = \T_{i_1, \dots, i_d}$.
Let $\T \in V^\perp \cap (\Q \cdot V^\perp)$. Then $\T = \Q \cdot \cS$ for some~$\cS \in V^\perp$, so
$\T = \sum_{\bi \in \cI} \T_\bi e_{i_1} \otimes \cdots \otimes e_{i_d} = \sum_{\bi \in \cI} \cS_\bi q^{(1)}_{i_1} \otimes \cdots \otimes q^{(d)}_{i_d}$.
Therefore, the $k$-th principal flattening of $\T$ is
\begin{align*}
    \T^{(k)} &= e_0 \otimes \left( \sum_{\bi \in \cI \mid i_k = 0}  \T_\bi \bigotimes_{l \neq k} e_{i_l}\right) + e_1 \otimes \left( \sum_{\bi \in \cI \mid i_k = 1}  \T_\bi \bigotimes_{l \neq k} e_{i_l}\right) \\
    &= q^{(k)}_0 \otimes \left( \sum_{\bi \in \cI \mid i_k = 0}  \cS_{\bi} \bigotimes_{l \neq k} q^{(l)}_{i_l}\right) + q^{(k)}_1 \otimes \left( \sum_{\bi \in \cI \mid i_k = 1}  \cS_{\bi} \bigotimes_{l \neq k} q^{(l)}_{i_l}\right).
\end{align*}
Both expressions above are SVDs of the matrix $\T^{(k)}$, because every pair of distinct strings~$\bi ,\bj$ in $\cI$ is at Hamming distance at least two for $d \geq 4$, so the sets
\begin{align*}
    &\{ (i_1, \dots, i_{k-1}, i_{k+1}, \dots, i_d) \in \{ 0,1\}^{d-1} \mid (i_1, \dots, i_{k-1},0, i_{k+1}, \dots, i_d) \in \cI\} \text{ and} \\
    & \{ (i_1, \dots, i_{k-1}, i_{k+1}, \dots, i_d) \in \{ 0,1\}^{d-1} \mid (i_1, \dots, i_{k-1},1, i_{k+1}, \dots, i_d) \in \cI\}
\end{align*}
are Hamming distance at least one for $d \geq 4$.  Hence both $e_{0}, e_1$ and~$q^{(k)}_0, q^{(k)}_1$ are left singular vectors of $\T^{(k)}$. The SVD is unique unless two singular values are equal. So the two singular values, $\sqrt{\sum_{\bi \in \cI \mid i_k = 0} \T_\bi^2}$ and $\sqrt{\sum_{\bi \in \cI \mid i_k = 1} \T_\bi^2}$, are equal. This is equivalent to $f_k(\T) = 0$.

Next, we show that the variety defined by the ideal $I \coloneqq (f_1, \dots, f_d) \subset R \coloneqq \R[\T_{\bi} \mid \bi \in \cI]$ is a complete intersection; i.e., each equation reduces the dimension by one.  Given $k \in [d]$, let~$\chi(k) = (\chi_1, \dots, \chi_d) \in \{0,1\}^{d}$ be such that $\chi_i$ is one if $i = k$ and zero otherwise, and let~$\overline{\chi}(k) = (\overline{\chi}_1, \dots, \overline{\chi}_d) \in \{0,1\}^{d}$ be such that $\chi_i$ is zero if $i = k$ and one otherwise. Note that~$\cI = \{\chi(k), \overline{\chi}(k) \mid k \in [d]\}$.
We can rewrite $f_k$ as
    \[
    f_k = -(\T_{\chi(k)}^2 - \T_{\overline{\chi}(k)}^2) + \sum_{l \neq k} (\T_{\chi(l)}^2 - \T_{\overline{\chi}(l)}^2). 
    \]
    Hence, for all $k \in [d]$ we have
    \[
    -(d-3) f_k + \sum_{l \neq k} f_l = 2(d-2) \left( \T_{\chi(k)}^2 - \T_{\overline{\chi}(k)}^2\right).
    \]
    Therefore, $(f_1, \dots, f_d) = (\T_{\chi{(1)}}^2 - \T_{\overline{\chi}{(1)}}^2, \dots, \T_{\chi{(d)}}^2 - \T_{\overline{\chi}{(d)}}^2)$ and, since each generator on the right-hand side involves different variables, we have $\dim(R/I) = 2d-d=d$.

If there are exactly $r > 0$ matrices $Q_{k_1}, \dots, Q_{k_r} \notin \rSP(2)$ and $\T \in (\Q \cdot V^\perp) \cap V^\perp$, then $ f_{k_1}(\T) =  \cdots = f_{k_r}(\T) = 0$, so $\dim((\Q \cdot V^\perp) \cap V^\perp) \leq \dim(R/(f_{k_1}, \dots, f_{k_r})) = 2d - r$.
\end{proof}

\begin{remark}
    The fact that every pair of binary strings in $\cI$ is at Hamming distance at least $2$ for $d \geq 4$ played a crucial role in the last proof, since we could argue about the SVD of the principal flattenings of tensors in $V^\perp$. To visualize this, let $d = 4$ and $\T \in V^\perp$. Then,
    \begin{align*}
        \T^{(1)} &= \begin{pmatrix}
        0 & \T_{0001} & \T_{0010} & 0 & \T_{0100} & 0 & 0 & \T_{0111} \\
        \T_{1000} & 0 & 0 & \T_{1011} & 0 & \T_{1101} & \T_{1110} & 0
    \end{pmatrix} \\
    & = \begin{pmatrix}
        1 & 0 \\
        0 & 1
    \end{pmatrix}
    \begin{pmatrix}
        \sigma_0 & 0 \\
        0 & \sigma_1
    \end{pmatrix}
    \begin{pmatrix}
        0 & \frac{\T_{0001}}{\sigma_0} & \frac{\T_{0010}}{\sigma_0} & 0 & \frac{\T_{0100}}{\sigma_0} & 0 & 0 & \frac{\T_{0111}}{\sigma_0} \\
        \frac{\T_{1000}}{\sigma_1} & 0 & 0 & \frac{\T_{1011}}{\sigma_1} & 0 & \frac{\T_{1101}}{\sigma_1} & \frac{\T_{1110}}{\sigma_1} & 0
    \end{pmatrix}, 
    \end{align*}
    where $\sigma_0 = \sqrt{\sum_{\bi \in \cI \mid i_k = 0} \T_\bi^2}$ and $\sigma_1 = \sqrt{\sum_{\bi \in \cI \mid i_k = 1} \T_\bi^2}$, is an SVD of $\T^{(1)}$.
\end{remark}

Lemma~\ref{lem:codimension-not-in-SP} gives us a way to show \eqref{eq:positive-codimension}. Before we prove \Cref{thm:main} for binary tensors, it remains to consider the case $\Q \in \rSP(2)^d$.  Given $\Q \in \rSP(2)^{d}$ we define the binary string $\bi(\Q) = (i_1, \dots, i_d) \in \{ 0,1\}^{d}$ by 
\begin{equation} \label{eq:iQ}
    i_k = \begin{cases}
    0 &\text{if } Q_k = \begin{pmatrix}
        \pm 1 & 0 \\ 0 & \pm 1 
    \end{pmatrix}\\[1em]
    1 &\text{if } Q_k = \begin{pmatrix}
        0 & \pm 1 \\\pm 1 & 0
    \end{pmatrix}.
\end{cases}
\end{equation}
The $k$-th coordinate of $\bi(\Q)$ encodes whether $\Q$ permutes $e_0$ and $e_1$ on the $k$-th factor. The sum $|\bi(\Q)|$ is the number of factors in which $e_0$ and $e_1$ are permuted. Using this notation and
\Cref{thm:main-2222}, \Cref{prop:main-equivalence} specializes to binary tensors as follows.

\begin{proposition}\label{prop:main-equivalence-binary}
    \Cref{thm:main} for $(\R^2)^{\otimes d}$ is equivalent to the following statement: Consider a generic $ \T \in V \subset (\R^{2})^{\otimes d}$ and let $\Q \in \rO(2)^d$. Then $\Q \cdot \T \in V$ if and only if $\Q \in \rSP(2)^d$, and $|\bi(\Q)| \in \{ 0, d\}$ for $d \neq 4$ or $|\bi(\Q)| \in \{ 0, 2, 4 \}$ for $d = 4$.
\end{proposition}

\begin{lemma} \label{lem:dim-QinSP}
    Fix $\Q \in \rSP(2)^d$ with $d \geq 4$. Then,
    \[
    \codim((\Q \cdot V) \cap V, V) =
    \begin{cases}
        0 &\text{if } |\bi(\Q)| \in \{ 0,d\} \text{ for } d \geq 5 \text{ or } |\bi(\Q)| \in \{ 0,2,4\} \text{ for } d = 4 \\
        2d-4 &\text{if } |\bi(\Q)| \in \{ 2,d-2\} \text{ for } d \geq 5 \\
        2d &\text{otherwise}.
    \end{cases}
    \]
\end{lemma}
\begin{proof}
We study rank of $M_\Q$, by \Cref{lem:MQ}. 
We have $[M_{\Q}]_{\bi, \bj} \neq 0$ if and only if $q^{(1)}_{i_1 j_1}\neq 0$, \ldots, $q^{(d)}_{i_d j_d}\neq 0$ or, equivalently, $\bj = \bi + \bi(\Q)$ (where the sum is elementwise and modulo two so that the result is a binary string).
Define
    \[
    \cI + \bi(\Q) \coloneqq \{\bi + \bi(\Q) \mid \bi \in \cI\} = \{\bj \in \{0,1\}^{d} \mid |\bi(\Q) + \bj| \in \{ 1, d-1\}\}.
    \]
Then $M_\Q$ has $|(\cI + \bi(\Q)) \cap (\{0,1\}^{d} \setminus \cI)|$ nonzero rows and each column of $M_{\Q}$ has at most one nonzero entry. 

If $|\bi(\Q)| \in \{ 0, d\}$ and $d \geq 5$ or $|\bi(\Q)| = \{ 0, 2, 4 \}$ and $d = 4$, then $\cI + \bi(\Q) = \cI$, so $M_\Q$ is the zero matrix.  If $|\bi(\Q)| \in \{ 2, d-2\}$ and $d \geq 5$, then $(\cI + \bi(\Q))\cap \cI$ has size four, so $\dim( \ker M_\Q) = 4$. 
If $|\bi(\Q)| \notin \{ 0, 2, d-2, d\}$, then $|\bj| \notin \{1, d-1\}$ for all $\bj \in \cI + \bi(\Q)$, so $M_\Q$ has full rank.
\end{proof}

\begin{remark}
    To visualize why $d=4$ is special, consider a binary order-four tensor $\T \in V$ and let $\Q = (P,P,I,I)$ where $I$ and $P$ are the identity and transposition, respectively, as in \Cref{ex:MQd4}.
   The entries of  $\T$, and those of its image under $(P, P, I, I)$ are
\[
\T = \left(\begin{array}{cc|cc}
     \T_{0000}&  0& 0& \T_{0101}\\
     0 & \T_{1100}& \T_{1001}& 0\\
     \hline
     0 & \T_{0110} & \T_{0011} & 0\\
     \T_{1010} & 0 & 0 & \T_{1111}
\end{array}\right), \quad (P, P, I, I) \cdot \T = \left(\begin{array}{cc|cc}
     \T_{1100}&  0& 0& \T_{1001}\\
     0 & \T_{0000}& \T_{0101}& 0\\
     \hline
     0 & \T_{1010} & \T_{1111} & 0\\
     \T_{0110} & 0 & 0 & \T_{0011}
\end{array}\right) .
\]
The sparsity pattern of the two matrices is the same. Hence, the image lies in $V$. In this case, $|\bi(\Q)| =2$ and $M_\Q = 0 \in \R^{8 \times 8}$.
\end{remark}
Next, we offer a refinement of \Cref{lem:codimension-not-in-SP}.
\begin{lemma} \label{lem:generic-codimension}
    Let $d \geq 4$.
    Let $\Q \in \rO(2)^{d}$ with exactly $r > 0$ matrices~$Q_k$ such that $Q_k \notin \rSP(2)$ but otherwise generic, then $\codim((\Q \cdot V) \cap V, V) = 2d$. In particular, for generic $\Q \in \rO(2)^{d}$ we have $\codim((\Q \cdot V) \cap V, V) = 2d$.
\end{lemma}

\begin{proof}
    For $r \in [d]$ let $W_r$ be the set of tuples $\Q = (Q_1, \dots, Q_d) \in \rO(2)^{d}$ with exactly $r$ matrices $Q_k$ not in $\rSP(2)$. For each $r \in [d]$ and $P_{1}, \dots, P_{d-r} \in \rSP(2)$, there exists $\Q \in \rSP(2)^d$ with $Q_k = P_k$ for all $1 \leq k \leq d-r$ (if $r=d$ there are no conditions imposed) and $|\bi(\Q)| \notin \{ 0,2,d-2,d\}$. For this $\Q$,  $M_\Q$ has full rank by \Cref{lem:dim-QinSP}. Hence, the matrix $M_\Q$ has full rank for generic $\Q \in W_r$, because $\rSP(2)^d$ is contained in the closure of $W_r$ for all $r \in [d]$ and the rank of a matrix is a lower semicontinuous function. Equivalently, $\codim((\Q \cdot V) \cap V, V) = 2d$ for generic $\Q \in W_r$ for all $r \in [d]$, by \Cref{lem:MQ}.
\end{proof}

We are now ready to prove our main result for binary tensors.

\begin{proof}[Proof of \Cref{thm:main} for $(\R^2)^{\otimes d}$]
The case $d=2$ is the generic uniqueness of the singular value decomposition. The case $d=3$ is \Cref{prop:main-binary-d3}. 
Let $d \geq 4$. We prove the equivalent statement in Proposition~\ref{prop:main-equivalence-binary}. First, suppose that $\Q \in \rSP(2)^d$. Then $\codim((\Q \cdot V^\perp) \cap V^\perp, V^\perp) > 0$ unless $|\bi(\Q)| \in \{0,d \}$ for $d \geq 5$ or $|\bi(\Q)| \in \{0,2,4 \}$ for $d = 4$, by \Cref{lem:dim-QinSP}. This shows that, to prove the statement in Proposition~\ref{prop:main-equivalence-binary},  it suffices to show \eqref{eq:positive-codimension} or, equivalently, by \Cref{lem:MQ}, that
\[
\codim \left( \bigcup_{\Q \in \rO(2)^{d} \setminus \rSP(2)^d} (\Q \cdot V^\perp) \cap V^\perp, V^\perp \right) > 0.
\]
Consider the variety $W$ defined  by the $2d\times 2d$ minors of $M_\Q$
\[
W \coloneqq \{\Q \in \rO(2)^{d} \mid \rank(M_\Q) < 2d\} = \{\Q \in \rO(2)^{d} \mid (\Q \cdot V^\perp) \cap V^\perp \neq \{0\}\}.
\]
Then
\[
\bigcup_{\Q \in \rO(2)^d \setminus \rSP(2)^d} (\Q \cdot V^\perp) \cap V^\perp = \bigcup_{\Q \in W \setminus \rSP(2)^d} (\Q \cdot V^\perp) \cap V^\perp, 
\]
and $\dim(W) < \dim(\rO(2)^{d}) = d$, by Lemmas \ref{lem:MQ} and \ref{lem:generic-codimension}. For $r \in [d]$ let $W_r$ be the set of tuples $\Q = (Q_1, \dots, Q_d) \in \rO(2)^{d}$ for which exactly $r$ matrices $Q_k$ are not in $\rSP(2)$.
We can decompose $W \setminus \rSP(2)^d$ as $W \setminus \rSP(2)^d = \cup_{r=1}^d (W \cap W_r) \setminus \rSP(2)^d$. Let
$\Q \in W_r$, then $\codim((\Q \cdot V^\perp) \cap V^\perp, V^\perp) \geq r$ by \Cref{lem:codimension-not-in-SP}. We have that $\dim(W_r) = r$ and, for a generic tuple $\Q \in W_r$, $(\Q \cdot V^\perp) \cap V^\perp = \{0\}$, by \Cref{lem:generic-codimension} and \Cref{lem:MQ}. 
Hence, $\dim(W \cap W_r) < r$. 
Therefore, for all $r \in [d]$ we have
\[
\codim\left(\bigcup_{\Q \in (W \cap W_r) \setminus \rSP(2)^d} (\Q \cdot V^\perp) \cap V^\perp, V^\perp \right) >0 . \qedhere
\]
\end{proof}

\Cref{thm:main-2222} refines  \Cref{thm:main} for binary tensors of order 4.

\begin{proof}[Proof of Theorem~\ref{thm:main-2222}]
Consider a generic $\T \in V \subset \R^{2 \times 2 \times 2 \times 2}$.
The orthogonal bases of singular vector tuples of $\T$ are $\{\pm e_{i_1} \otimes e_{i_2} \otimes e_{i_3} \otimes e_{i_4}, \pm e_{i_1+1} \otimes e_{i_2+1} \otimes e_{i_3+1} \otimes e_{i_4+1}\}$ with $|\bi| \in \{ 0, 2, 4\}$, where the sums in the indices are taken modulo 2, by \Cref{prop:main-equivalence-binary}. Equivalently, $\T$ has four distinct orthogonal basis of singular vector tuples (up to sign flip):
\begin{align*}
    &\{ e_0 \otimes e_0 \otimes e_0 \otimes e_0, e_0 \otimes e_0 \otimes e_0 \otimes e_0\}, \quad \{ e_0 \otimes e_0 \otimes e_1 \otimes e_1, e_1 \otimes e_1 \otimes e_0 \otimes e_0\}, \\
    &\{ e_0 \otimes e_1 \otimes e_0 \otimes e_1, e_1 \otimes e_0 \otimes e_1 \otimes e_0\}, \quad \{ e_0 \otimes e_1 \otimes e_1 \otimes e_0, e_1 \otimes e_0 \otimes e_0 \otimes e_1\}. \qedhere
\end{align*}
\end{proof}

\subsection{The general case $n_1 \times \cdots \times n_d$}

We proved \Cref{thm:main} for binary tensors in the last section.
The idea of the general proof is to reduce the general case to the binary case using the normal form of orthogonal transformations. Any $Q\in \rO(n)$ can be written as~$Q=P R P^\top$ where $P\in \rO(n)$ and $R$ has the following block-diagonal form:
\begin{equation}\label{eq:orthogonal-blocks}
R=\begin{pmatrix}
    B_1 &     &\\
        & \ddots &\\
        &        & B_{\frac{n}{2}}\\
\end{pmatrix} (n \text{ even}), 
\quad R=\begin{pmatrix}
    B_1 & & & \\
        & \ddots & & \\
        &        & B_{\frac{n-1}{2}}& \\
        & & & \pm 1
\end{pmatrix} (n \text{ odd}),
\end{equation}
with $B_i \in \rO(2)$ for all $i$. See, e.g., \cite[Theorem~10.19]{roman2005advanced}. 
We will use the following lemma to show that, up to an orthogonal change of basis, the tuple $\Q$ 
is of the form \eqref{eq:orthogonal-blocks}.

\begin{lemma}[Change of basis]\label{lem:reduction}
 Let $\T \in \R^{n_1 \times \cdots \times n_d}$ and let $\cS = \bP \cdot \T$, where $P_k \in \rO(n_k)$. Then $\T \in V$ and $\Q \cdot \T \in V$ if and only of $\cS \in \bP \cdot V$ and $\bR \cdot \cS \in \bP \cdot V$, where $R_k = P_k Q_k P_k^\top$. The coordinates of $\cS$ and $\bR \cdot \cS$ in the basis $\{p_{i_1}^{(1)} \otimes \cdots \otimes p_{i_d}^{(d)} \mid i_k \in [n_k]\}$ are $\T_{i_1, \dots, i_d}$ and $[\Q \cdot \T]_{i_1, \dots, i_d}$, respectively, where $p_i^{(k)}$ is the $i$-th column of $P_k$.
\end{lemma}

\begin{proof}
    Let $\bP^\top = (P_1^\top, \dots, P_d^\top)$.
    We have $\Q \cdot \T=\bP^\top \cdot(\bR \cdot \cS)$, $\bR \cdot \cS = \bP \cdot(\Q \cdot \T)$, and~$\T=\bP^\top \cdot \cS$. For the coordinates, write
    $$
    \T=\sum_{i_1, \dots, i_d} \T_{i_1,\ldots,i_d}  e_{i_1}\otimes \cdots \otimes e_{i_d} \;\;\;\mbox{ and }\;\;\; \Q\cdot \T=\sum_{i_1, \dots, i_d} [\Q \cdot \T]_{i_1,\ldots,i_d}  e_{i_1}\otimes \cdots \otimes e_{i_d}.
    $$
    By linearity of the group action, acting with $\bP$ on both expressions gives
    \[
    \cS\;=\;\sum_{i_1, \dots, i_d} \T_{i_1,\ldots,i_d} \bP\cdot e_{i_1}\otimes \cdots \otimes e_{i_d} \;\;\;\mbox{ and }\;\;\; \bR \cdot \cS\;=\;\sum_{i_1, \dots, i_d} [\Q \cdot \T]_{i_1,\ldots,i_d} \bP\cdot e_{i_1}\otimes \cdots \otimes e_{i_d}. \qedhere
    \]
\end{proof}

\begin{lemma} \label{lem:finite-SVTs}
    Consider the linear space $U = \{ \T \in (\R^2)^{\otimes d} \mid \T_{2,1, \dots, 1} = \cdots = \T_{1, \dots, 1, 2} = 0\}$. Let $\T \in U$ be generic, then $\Q \cdot \T \in U$ for finitely many $\Q \in \rO(2)^d$.
\end{lemma}

\begin{proof}
    The linear space $U$ consists of tensors for which $e_1^{\otimes d}$ is a singular vector tuple. Hence, $\rO(2)^d \cdot U = (\R^2)^{\otimes d}$, because every tensor has a singular vector tuple: every tensor has a best rank-one approximation. Hence, we have a parametrization $\phi: \rO(2)^d \times U \to (\R^2)^{\otimes d}, (\Q, \T) \mapsto \Q \cdot \T$. We have $\dim\rO(2)^{d} = d$, $\dim U = 2^d - d$ and $\dim (\R^2)^{\otimes d} = 2^d$. Therefore, $\phi$ is generically finite-to-one, by the fiber dimension theorem (see, e.g., \cite[Chap.~1, \S~8, Thm.~3]{mumford2004red}). Finally, given a generic $\T = \Q \cdot \cS$ with $\Q \in \rO(2)^d$ and $\cS \in  U$, $\T$ has finitely many singular vector tuples if and only if $\cS$ does so.
\end{proof}

The following result shows that the statement in \Cref{prop:main-equivalence} holds for tuples of block-diagonal matrices as in \eqref{eq:orthogonal-blocks}.

\begin{lemma} \label{lem:reduction-normal-form}
    Let $d\geq 3$, let $n_d \geq 3$, let $\T \in V$ generic, and let $\Q \in \rO(n_1) \times \cdots \times \rO(n_d)$ such that $Q_k$ has the block form in \eqref{eq:orthogonal-blocks} for all $k \in [d]$. Then, $\Q \cdot \T \in V$ if and only if
    \[
    Q_1 \in \rSP(n_1) \quad \text{ and } \quad Q_k = \begin{pmatrix}
        Q_1D_k & 0 \\
        0 & \tilde{Q}_k
    \end{pmatrix} \quad \text{for all } k \geq 2,
    \]
    where $D_k$ is a diagonal matrix with diagonal entries $\pm 1$ and $\tilde{Q}_k \in \rO(n_k - n_1)$.
\end{lemma}

\begin{proof}
For each $k \in [d]$, let $B_1^{(k)} \in \rO(2)$ be the upper-left $2\times 2$ block of $Q_k$. 
    Given any tuple $\bi \in \{1,2\}^{d}$ we have 
    \[
    [\Q \cdot \T]_{\bi} = \sum_{j_1, \dots, j_d = 1}^2 q^{(1)}_{i_1, j_1} \cdots q^{(d)}_{i_d, j_d}
    \T_{j_1, \dots, j_d} = \sum_{j_1, \dots, j_d = 1}^2 [B_1^{(1)}]_{i_1,j_1} \cdots [B_1^{(d)}]_{i_d,j_d} \T_{j_1, \dots, j_d}.
    \]
This formula involves only the entries of $\Q$ that lie in the $2\times 2$ blocks $\bs B^{(1)}=(B_1^{(1)}, \ldots, B_1^{(d)})$ and only the corresponding $2\times \cdots\times 2$ subtensor of the tensor $\cT$. Letting $A=\{1,2\}^{d}$, denote this subtensor by $\cT_A$. The above formula gives $(\Q\cdot \cT)_A=\bs B^{(1)}\cdot \cT_A$ and it allows us to use the binary case proven earlier.

When $d\neq 4$, the genericity of $\T$ (and so $\cT_A$) together with the conditions $[\Q \cdot \T]_{\bi}=0$ for all $\bi \in \{1,2 \}^{d}$ such that $d_H(\bi, (1, \dots, 1)) = 1$ or $d_H(\bi, (2, \dots, 2)) = 1$ implies that
$B_1^{(1)} \in {\rm SP}(2)$ and, for all $2 \leq k \leq d$, $B_1^{(k)} = B_1^{(1)} \diag (\pm 1, \pm 1)$.
    When $d = 4$, applying the binary case directly, we get $|\bi((B_1^{(1)}, \dots , B_1^{(4)}))| \in \{0, 2, 4\}$ (c.f \eqref{eq:iQ}).
    We show that $|\bi((B_1^{(1)}, \dots , B_1^{(4)}))| = 2$ cannot happen for generic $\T$. Suppose $B_1^{(1)} = B_1^{(2)} = I_2$ and $B_1^{(3)} = B_1^{(4)} = \begin{pmatrix}
    0 & 1 \\ 1 & 0
\end{pmatrix}
$.
Then,
\[
[(Q_1, \dots, Q_4) \cdot \T]_{1,1,1,3} = \begin{cases}
    \pm \T_{1,1,2,3} & \text{if } n_4 = 3 \text{ or } [B_2^{(4)}]_{1,1} = \pm 1\\
    \pm \T_{1,1,2,4} & \text{if } [B_2^{(4)}]_{1,1} = 0.
\end{cases}
\]
This is nonzero for generic $\T$. Hence $|\bi((B_1^{(1)}, \dots, B_1^{(4)}))| \in \{0,4\}$. The proof for the other tuples with $|\bi((B_1^{(1)}, \dots , B_1^{(4)}))| = 2$ is analogous.
This concludes the proof if $n_1$ is even, since $Q_1$ consists of $n_1/2$ diagonal $2\times 2$ blocks. 

Suppose now $n_1$ is odd (hence $n_1\geq 3$). The last entry of $Q_1$ does not belong in a $2 \times 2$ block. We show that $q^{(k)}_{n_1,n_1} = \pm 1$ for all $k \geq 2$. If $n_1 = \cdots = n_d$, then $q^{(k)}_{n_1n_1} = \pm 1$, by construction.
Let $n_1  = \cdots = n_{s-1}< n_s \leq \cdots \leq n_d$ for $1 < s \leq d$. Then $q^{(k)}_{n_1, n_1} = \pm 1$ for all $1 \leq k < s$. If $s = d$, then $[\Q \cdot \T]_{1, n_1, \dots, n_1}$ is 
$\pm q^{(d)}_{n_1,n_1 + 1} \T_{1, n_1, \dots,n_1, n_1 + 1}$ or $q^{(d)}_{n_1,n_1 + 1} \T_{2, n_1, \dots,n_1, n_1 + 1}$, depending on $Q_1$.
This is zero for generic $\T$ if and only if $q^{(d)}_{n_1,n_2} = 0$, and so $q^{(d)}_{n_1,n_1} = \pm 1$. Now, suppose that $s < d$. Let $\T \in V$ be generic, and suppose that $\Q \cdot \T \in V$. Consider the 
blocks indexed by $\{ n_1\}^{s-1} \times \{ n_1, n_1 + 1\}^{d-s+1}$. Then $\Q \cdot \T \in V$ implies $[\Q \cdot \T]_{n_1, \dots, n_1, {n_1 + 1}, n_1, \dots, n_1} = 0$. This system of equations has finitely many solutions $q^{(k)}_{i,j}$ for $i,j \in \{n_1, n_1 + 1\}$ and $s \leq k \leq d$, by genericity of $\T$ and \Cref{lem:finite-SVTs}. We show that the extra condition 
$[\Q \cdot \T]_{1, n_1, \dots, n_1} = 0$
is enough to conclude the proof.
The top-left block of $Q_1$ is in $\rSP(2)$. Suppose that $q^{(1)}_{11} = \pm 1$, the case with $q^{(1)}_{12} = \pm 1$ is analogous. Then,
\begin{equation} \label{eq:extra-equation}
    [\Q \cdot \T]_{1, n_1, \dots, n_1} = \sum_{i_s, \dots, i_d = n_1}^{n_1 + 1} \pm q^{(s)}_{n_1, i_s} \cdots q^{(d)}_{n_1, i_d} \T_{1, n_1, \dots, n_1, i_s, \dots, i_d} = 0.
\end{equation}
This equation is independent from the equations $[\Q \cdot \T]_{n_1, \dots, n_1, {n_1 + 1}, n_1, \dots, n_1} = 0$, 
since it involves different entries of $\T$. Hence, viewing \eqref{eq:extra-equation} as a polynomial in $\Q$, every monomial must be zero by genericity of $\T$. 
The polynomial \eqref{eq:extra-equation} for generic $\T \in V$ has nonzero coefficient in all monomials $q^{(s)}_{n_1, i_s} \cdots q^{(d)}_{n_1, i_d}$ for $i_s, \dots, i_d \in \{ n_1, n_1 + 1\}$,
except for $q^{(s)}_{n_1, n_1} \cdots q^{(d)}_{n_1, n_1}$ because $\T_{1, n_1, \dots, n_1} = 0$. Therefore, all these monomial are zero if and only if $q^{(k)}_{n_1, n_1 + 1} = 0$ for all $s \leq k \leq d$. This is equivalent to $q^{(k)}_{n_1, n_1} = \pm 1$ for all $s \leq k \leq d$.
\end{proof}

Using Lemmas \ref{lem:reduction} and \ref{lem:reduction-normal-form}, we reduce the general case to the block-diagonal form.

\begin{proof}[Proof of \Cref{thm:main}]
    The case $d=2$ follows from the generic uniqueness of the singular value decomposition.
    The binary case was proved in the previous section.
    Let $d\geq 3$ and $n_d \geq 3$. Let $\T \in V$ be generic and suppose that $\T$ has two orthogonal bases of singular vector tuples. That is, there exists $\Q = (Q_1, \dots, Q_d) \in \rO(n_1) \times \cdots \times \rO(n_d)$ such that $\Q \cdot \T \in V$ and the sets $\{e_j^{\otimes d} \mid j \in [n_1] \}$ and $\{q_{j}^{(1)} \otimes \cdots \otimes q_{j}^{(d)} \mid j \in [n_1]\}$ are orthogonal bases of singular vector tuples. We show that these two sets are equal up to sign flips.
    
    For each $k \in [d]$, we write
    $Q_k = P_k^\top R_k P_k$, where $P_k \in \rO(n_k)$ and $R_k$ has block-diagonal structure \eqref{eq:orthogonal-blocks}. Let $\bP = (P_1, \dots, P_d)$, $\bR = (R_1, \dots, R_d)$, and $\bP^{\top} = (P_1^\top, \dots, P_d^\top)$. 
    
    Let $\cS = \bP \cdot \T \in \bP \cdot V$.  Then $\{ \bP \cdot e_j^{\otimes d} \mid j \in [n_1]\} = \{ p_j^{(1)} \otimes \cdots \otimes p_j^{(d)} \mid j \in [n_1]\}$ are singular vector tuples of $\cS$,
    by \Cref{lem:equivariance-singular-vectors}.
    Moreover, we have that $\bR \cdot \cS = \bP \cdot(\Q \cdot \T) \in \bP \cdot V$. So~$\{ \bP \cdot q_j^{(1)} \otimes \cdots \otimes q_j^{(d)} \mid j \in [n_1]\} = \{ \bR \cdot p_j^{(1)} \otimes \cdots \otimes p_j^{(d)} \mid j \in [n_1]\}$ are singular vector tuples of $\cS$, by \Cref{lem:equivariance-singular-vectors}.
        Since $\T$ is generic  in $V$, $\cS$ is generic in $\bP \cdot V$. Hence $\{ p_j^{(1)} \otimes \cdots \otimes p_j^{(d)} \mid j \in [n_1]\}$ and $\{ \bR \cdot p_j^{(1)} \otimes \cdots \otimes p_j^{(d)} \mid j \in [n_1]\}$ are the same set of singular vector tuples of $\cS$ up to sign flips, 
    using \Cref{lem:reduction} together with \Cref{lem:reduction-normal-form} in the basis $\{ p_{j_1}^{(1)} \otimes \cdots \otimes p_{j_d}^{(d)}\} \mid j_k \in [n_k] \}$.
    Applying $\bP^\top$ we see that $\{ e_j^{\otimes d} \mid j \in [n_1]\}$ and~$\{q_j^{(1)} \otimes \cdots \otimes q_j^{(d)} \mid j \in [n_1]\}$ are the same set of singular vector tuples of $\bP^\top \cdot \cS = \T$ up to sign flips.
\end{proof}

\section{Uniqueness of orthogonal eigenvectors} \label{sec:proof-uniqueness-symmetric}

In this section, we prove \Cref{thm:main-symmetric}. 
As in the non-symmetric case, we start with binary tensors: we want to show that $\codim ((Q \bullet V_\sym) \cap V_\sym, V_\sym)$ is large enough for~$Q \notin \rSP(2)$ so that $\bigcup_{Q \notin \rSP(2) } (Q \bullet V_\sym) \cap V_\sym$ has positive codimension in $V_\sym$.
Specializing the proof of \Cref{prop:main-equivalence} to symmetric tensors we get the following.

\begin{proposition} \label{prop:main-equivalence-symmetric}
    \Cref{thm:main-symmetric} is equivalent to the following statement: If $\T \in V_\sym$ is generic and $Q \in \rO(n)$, then $Q \bullet \T \in V_\sym$ if and only if $Q \in \rSP(n)$.
\end{proposition}

The statement in \Cref{prop:main-equivalence-symmetric} is precisely \cite[Conjecture 5.17]{mesters2022non}.
Denote the canonical basis of $\R^2$ by $\{e_0, e_1\}$. Recall that the coordinates of a tensor $\T \in (\R^2)^{\otimes d}$ are given by binary strings of length $d$. If $\T$ is symmetric, then its entry $\T_\bi$ is determined by $|\bi| = \sum_{k} i_k$.
    Hence, a symmetric tensor in $S^d(\R^2)$ is specified by $d+1$ parameters $t_0, \dots, t_d$, where $\T_\bi = t_k$ for all $\bi \in \{0,1\}^{d}$ such that $|\bi| = k$. We will follow this notation to specify entries of symmetric binary tensors. This way we have that
    \[
    V_\sym = \{ \T \in S^d(\R^2) \mid t_1 = t_{d-1} = 0\},
    \]
    and its orthogonal complement is
    \[
    V_\sym^\perp = \{ \T \in S^d(\R^2) \mid t_0 = t_2 = \cdots = t_{d-2} = t_d = 0\}.
    \]
The following matrices play a special role in computing $\codim((Q \cdot V_\sym) \cap V_\sym, V_\sym)$:
\begin{equation}\label{eq:eight}
    \pm\frac{1}{\sqrt{2}}\begin{pmatrix}
    1 & 1\\
    1 & -1
\end{pmatrix}, \pm\frac{1}{\sqrt{2}}\begin{pmatrix}
    1 & 1\\
    -1 & 1
\end{pmatrix}, \pm\frac{1}{\sqrt{2}}\begin{pmatrix}
    -1 & 1\\
    1 & 1
\end{pmatrix}, \pm\frac{1}{\sqrt{2}}\begin{pmatrix}
    1 & -1\\
    1 & 1
\end{pmatrix}.
\end{equation}
\begin{lemma}\label{lem:binaryex}
    Consider $V_\sym \subset S^d(\R^2)$ with $d \geq 3$ and let $Q \in \rO(2)$. Then,
    \[
    \codim((Q \bullet V_\sym) \cap V_\sym, V_\sym) = \begin{cases}
        0 & \text{if } Q \in \rSP(2) \\
        1 & \text{if }  d \in \{4,6 \}  \text{ and } $Q$ \text{ is one of \eqref{eq:eight} }\\
        2 & \text{otherwise}.
    \end{cases}
    \]
\end{lemma}

\begin{proof} 
If $Q \in \rSP(2)$, then $(Q \bullet V_\sym) \cap V_\sym = V_\sym$, so $\codim((Q \bullet V_\sym) \cap V_\sym, V) = 0$. Suppose $Q \notin \rSP(2)$. Adapting the proof of \Cref{lem:MQ} to symmetric tensors we get that $\codim((Q \bullet V_\sym) \cap V_\sym, V_\sym) = 2 - \dim((Q \bullet V_\sym^\perp) \cap V_\sym^\perp)$. 

Let $\Q = (Q, \dots, Q) \in \rO(2)^{d}$ and consider the matrix $M_\Q$ defined in \Cref{lem:MQ}. Let $\cI = \{ \bi \in \{ 0,1\}^{d} \mid |\bi| \in \{1, d-1 \}\}$. The rows of $M_\Q$ are indexed by $\{0,1\}^{d} \setminus \cI$ and its columns are indexed by $\cI$. Take $d-1$ rows of $M_\Q$, one for each weight of binary strings in $\{0,1\}^{d} \setminus \cI$. The resulting $(d-1) \times 2d$ submatrix does not depend on the choice of rows. 
Summing the columns indexed by binary strings in $\cI$ with the same weight we get $\tilde{M}_Q \in \R^{(d-1) \times 2}$. We index $\tilde{M}_Q$ by the weight of the corresponding binary strings.
That is, its columns are indexed by $\{1,d-1 \}$ and its rows are indexed by $\{0,2,\dots, d-2,d \}$. Then, the $k$-th row of~$\tilde{M}_Q$ for~$k = 0,2,\dots,d-2,d$ is
\[
\begin{pmatrix}
    k q_{00}^{d-k} q_{10}^{k-1}q_{11} + (d-k) q_{00}^{d-k-1} q_{01} q_{10}^k, & k q_{01}^{d-k} q_{10} q_{11}^{k-1} + (d-k) q_{00} q_{01}^{d-k-1} q_{11}^k
\end{pmatrix}.
\]
We have $\codim((Q \bullet V_\sym) \cap V_\sym, V_\sym) = \rank(\tilde{M}_Q)$, adapting the proof of \Cref{lem:MQ} to symmetric tensors. Suppose that $\rank(\tilde{M}_Q) \leq 1$. 
The submatrix of $\tilde{M}_Q$ obtained by taking the rows indexed by $0$ and $d$ is
\[
d\begin{pmatrix}
    q_{00}^{d-1}q_{01} & q_{00}q_{01}^{d-1} \\
    q_{10}^{d-1}q_{11} & q_{10}q_{11}^{d-1}
\end{pmatrix}
\]
The $2 \times 2$ minor vanishes if $\rank(M_Q) \leq 1$, which gives
\begin{equation} \label{eq:dmin2}
    (d^2 q_{00}q_{01}q_{10}q_{11})(q_{00}^{d-2}q_{11}^{d-2} - q_{01}^{d-2}q_{10}^{d-2}) = 0
\end{equation}
No entry of $Q$ is zero, 
since $Q \notin \rSP(2)$, so the second factor above must vanish.
If $d$ is odd, \eqref{eq:dmin2} cannot hold because $Q$ is orthogonal (so $\det(Q)=\pm 1$) and in this case $\rank(\tilde{M}_Q) =  2$. If $d$ is even then \eqref{eq:dmin2} holds if and only if $Q$ is one of the eight cases in \eqref{eq:eight}. Indeed, under \eqref{eq:dmin2}, the Hadamard square of $Q$ has rank 1. Given that this Hadamard square is a doubly stochastic matrix (both rows and columns sum up to $1$), we conclude that $q_{11}^2=q_{12}^2=q_{21}^2=q_{22}^2=1/2$ giving the eight exceptional cases in \eqref{eq:eight}. Assume that $d$ is even and that
$$
Q\;=\;\frac{1}{\sqrt{2}}\begin{pmatrix}
    1 & 1\\
    1 & -1
\end{pmatrix}.
$$
Calculations for the remaining cases are analogous. We have
\[
\tilde{M}_Q^\top = 
\begin{pmatrix}
    1 & 0 & -1\\
    1 & 0 & -1
\end{pmatrix} \text{ for } d=4, 
\quad 
\tilde{M}_Q^\top  = \frac{1}{4}
\begin{pmatrix}
    3 & 1 & 0 & -1 & -3 \\
     3 & 1 & 0 & -1 & -3
\end{pmatrix} \text{ for } d=6,
\]
so $\rank(\tilde{M}_Q) = 1$. If $d \geq 8$, then $\rank(\tilde{M}_Q) = 2$, since the submatrix of $\tilde{M}_Q$ obtained by taking the rows indexed by $0$ and $3$ is
\[
\frac{1}{2^{d/2}}
\begin{pmatrix}
    d & d \\
    d-6 & 6-d
\end{pmatrix}. \qedhere
\]
\end{proof}

\begin{proof}[Proof of \Cref{thm:main-symmetric}]
First, let $n=2$. Then
\[
\codim \left(\bigcup_{Q \in \rO(2) \setminus \rSP(2) } (Q \bullet V_\sym) \cap V_\sym,  V_\sym \right) > 0,
\]
by Proposition~\ref{lem:binaryex}, 
because $\dim\rO(2)=1$. This proves Proposition~\ref{prop:main-equivalence-symmetric} and also Theorem~\ref{thm:main-symmetric} for binary tensors.
Now, let $n \geq 3$. Using \Cref{thm:main-symmetric} for binary tensors, we can specialize \Cref{lem:reduction-normal-form} to the symmetric case and find that the statement in \Cref{prop:main-equivalence-symmetric} is true when $Q$ is a block diagonal matrix as in \eqref{eq:orthogonal-blocks}. 
Considering $S^d(\R^n) \subset (\R^n)^{\otimes d}$, \Cref{lem:reduction} specializes to symmetric tensors as follows. Fix $Q\in \rO(n)$ and let $Q=P^\top R P$ with $P \in \rO(n)$. Let $\T\in S^d(\R^n)$ and let $\cS=P \bullet \T$. Then, $\T \in V_\sym$ and $Q\bullet \T\in V_\sym$ if and only if $\cS \in P\bullet V_\sym$ and $R\bullet \cS\in P\bullet V_\sym$. The coordinates of $\cS$ and $R \bullet \cS$ in the basis $\{p_{i_1} \otimes \cdots \otimes p_{i_d} \mid i_k \in [n] \}$ are $\T_{i_1,\ldots,i_d}$ and $(Q\bullet \T)_{i_1,\ldots,i_d}$ respectively. \Cref{thm:main-symmetric} follows, specializing the argument in the proof of \Cref{thm:main} to the symmetric case.
\end{proof}

\section{The varieties of tensors} \label{sec:varieties}

 In this section, we study the varieties of tensors with orthogonal eigenvectors or singular vectors. We compute the dimension of $X$ and $X_{\sym}$ using their parametrizations and the generic uniqueness results. Since they are not full-dimensional for $d\geq 3$, this implies that a general tensor does not have an orthogonal basis of singular vector tuples (or eigenvectors in the symmetric case). We also give an implicit description of (the closure of) $X_{\sym} \subseteq (\R^2)^{\otimes d}$.

\subsection{Dimension of $X$ and $X_{\rm sym}$} \label{sec:dimension}

We start with the symmetric case. Recall that $\dim S^d(\R^n) = \binom{n+d-1}{d}$. We find the dimension of $X_\sym \subseteq S^d({\R^n})$.

\begin{proposition} \label{thm:dimensionXsym}

    Let $X_\sym \subseteq S^d(\R^n)$ be the set of symmetric tensors with an orthogonal basis of eigenvectors. For $d = 2$, $\dim X_\sym = \binom{n+1}{2}$.
    For $d \geq 3$, $\dim X_\sym = \binom{n+d-1}{d}-\binom{n}{2}$.
\end{proposition}

\begin{proof}
When $d=2$, the set $V_\sym$ is defined by the $\binom{n}{2}$ constraints $T_{ij}=0$ for $i<j$. 
It is the set of diagonal matrices. 
The spectral theorem says that $X_\sym = S^2(\R^n)$. We have $\dim \rO(n)=\binom{n}{2}$, $\dim V_\sym=n$, and $\dim X_\sym =\binom{n+1}{2}$. 

When $d \geq 3$, $V_\sym$ has codimension $2\binom{n}{2}$ in $S^d(\R^n)$. 
So $\dim V_\sym = \dim S^d(\R^n) - n(n-1) = \binom{n+d-1}{d}-n(n-1)$, and $\dim\rO(n) = \binom{n}{2}$.
    The set $X_\sym \subset S^d(\R^n)$ is parametrized by 
    \[
    \phi: \rO(n) \times V_\sym \to X_\sym, (Q, \T) \mapsto Q \bullet \T.
    \]
    This parametrization is generically finite-to-one, by \Cref{thm:main-symmetric}. So, by the fiber dimension theorem, we get
    \[
    \dim X_\sym = \dim \rO(n) + \dim V_\sym = \binom{n+d-1}{d}-\binom{n}{2}. \qedhere
    \]
\end{proof}

\begin{remark}
    The variety of symmetric odeco tensors in $S^d(\R^n)$ has dimension $\binom{n+1}{2}$ \cite{robeva2016orthogonal}, so it is smaller than $X_\sym$.
\end{remark}

Consider now $X\subseteq \R^{n_1\times\cdots\times n_d}$. Recall that we assume $n_1 \leq \cdots \leq n_d$.

\begin{proposition} \label{thm:dimensionX}
    Let $X \subseteq \R^{n_1 \times \cdots \times n_d}$ be the set of tensors with an orthogonal basis of singular vector tuples. For $d=2$, $\dim X = n_1 n_2$. For $d \geq 3$, $\dim X = n_1 \cdots n_d - d \binom{n_1}{2}$.
\end{proposition}

\begin{proof}
    If $d=2$, then $X = \R^{n_1 \times n_2}$ by the singular value decomposition.
    If $d \geq 3$, then $\dim V = n_1 \dots n_d - \sum_{k=1}^d n_1 (n_k - 1)$.
    The set $X \subset \R^{n_1 \times \cdots \times n_d}$ is parametrized by
    \[
    \phi: \rO(n_1) \times \cdots \times \rO(n_d) \times V \to X, (Q_1, \dots, Q_d, \T) \mapsto (Q_1, \dots, Q_d) \cdot \T.
    \]
    Let $\T \in X$ be generic.
    Write $\T = \bP \cdot \cS$ with $\bP \in \rO(n_1) \times \cdots \times \rO(n_d)$ and $\cS \in V$. Note that $\T$ can also be written as $\T = \bP \cdot \Q^\top \cdot(\Q \cdot \cS)$ for $\Q \in \rO(n_1) \times \cdots \times \rO(n_d)$. By  \Cref{thm:main}, \Cref{prop:main-equivalence} and the genericity of $\T$, we have that $\Q \cdot \cS \in V$ if and only if $Q_1 \in \rSP(n_1)$ and $Q_k = \begin{pmatrix}
        Q_1 D_k & 0 \\
        0 & R_k
    \end{pmatrix}$ for some $D_k$ diagonal with diagonal entries $\pm 1$ and some $R_k \in \rO(n_k - n_1)$ for all $k \in [d]$. Therefore, the dimension of the fiber $\phi^{-1}(\T)$ is $\sum_{k=1}^d\dim \rO(n_k - n_1) = \sum_{k=1}^d \binom{n_k - n_1}{2}$. Hence, by the fiber dimension theorem, we get
    \begin{align*}
        \dim X & = \dim\rO(n_1) \times \cdots \times \rO(n_d) + \dim(V) - \dim \phi^{-1}(\T) = \\
                & = n_1 \cdots n_d + \sum_{k=1}^d \binom{n_k}{2}  - n_1(n_k - 1) - \binom{n_k - n_1}{2} = n_1 \cdots n_d - d \binom{n_1}{2}.
    \end{align*}
    The last equality follows because $ n_1(n_k - 1) + \binom{n_k - n_1}{2} - \binom{n_k}{2} $ counts the unordered pairs in~$[n_1] \subseteq [n_k]$, which equals $\binom{n_1}{2}$.
\end{proof}

Let $\overline{X}$ and $\overline{X_\sym}$ denote the Zariski closure of $X$ and $X_\sym$, respectively.

\begin{proposition} \label{prop:Xirredicible}
    Let $X \subset \R^{n_1 \times \cdots \times n_d}$ be the set of tensors with an orthogonal basis of singular vector tuples.
    Then $X = \left( \mathrm{SO}(n_1) \times \cdots \times \mathrm{SO}(n_d) \right) \cdot V$, so $\overline{X}$ is an irreducible variety.
\end{proposition}

\begin{proof}
    Let $\T = \Q \cdot \cS \in X$ with $\Q \in \rO(n_1) \times \cdots \times \rO(n_d)$ and $\cS \in V$. If $Q_k \in \rO(n_k) \setminus \mathrm{SO}(n_k)$ for some $k \in [d]$, we can flip the sign of the first column of $Q_k$ and the first $k$-th slice of $\cS$. That is, let $\tilde{I} = \mathrm{diag}(-1, 1,\dots, 1) \in \R^{n_k \times n_k}$, then $Q_k \tilde{I} \in \mathrm{SO}(n_k)$, $(I, \dots, I, \tilde{I}, I, \dots, I) \cdot \cS \in V$ and
    $
    \T = (Q_1, \dots, Q_k\tilde{I}, \dots, Q_d) \cdot \left( (I, \dots, \tilde{I}, \dots, I) \cdot \cS \right)
    $,
    because $\tilde{I}^2 = I$. Since this holds for all $k$, $\T \in (\mathrm{SO}(n_1) \times \cdots \times \mathrm{SO}(n_d)) \cdot V$. Hence, $X$ is the image of a product of irreducible varieties under a polynomial map, so $\overline{X}$ is irreducible.
\end{proof}

Specializing the proof of \Cref{prop:Xirredicible} to the symmetric case we get the following.

\begin{proposition} \label{prop:Xirreducible-symmetric}
    Let $X_\sym \subset S^d(\R^n)$ be the set of symmetric tensors with an orthogonal basis of eigenvectors. Then $X_\sym = \rSO(n) \bullet V_\sym$, so $\overline{X_\sym}$ is an irreducible variety.
\end{proposition}

\begin{remark}
    The sets $X$ and $X_\sym$ are closed with respect to the Euclidean topology. Let $\mathrm{cl}(X)$ denote the Euclidean closure of $X$. Then we have that $\T \in \mathrm{cl}(X)$ if and only if $\inf_{\Q \in \rO(n_1) \times \cdots \times \rO(n_d)} \mathrm{dist}(\Q \cdot \T, V) = 0$. This infimum is achieved at a minimum because we are optimizing a continuous function over a compact set. The argument for $X_\sym$ is analogous.
\end{remark}

\subsection{The variety $X_\sym$ for symmetric binary tensors} \label{sec:variety-binary}    
    Let $n=2$, $d \geq 3$ and consider $X_\sym \subset S^d(\R^2)$. It has codimension one, by \Cref{thm:dimensionXsym}. Therefore, its Zariski closure $\overline{X_\sym}$ is defined by one polynomial. In this section, we find this polynomial.
     We adopt the notation from \Cref{sec:proof-uniqueness-symmetric}: we specify the coordinates of a symmetric tensor $\T \in S^d(\R^2)$ with $d+1$ parameters $t_0, \dotsm, t_d$, where $\T_\bi = t_{|\bi|}$ for all $\bi \in \{ 0,1\}^{d}$. Throughout this section, let $\ri$ be the imaginary unit, i.e., $\ri^2 + 1 = 0$.
     
    \begin{theorem} \label{thm:variety-binary-symmetric}
        For $d \geq 3$, the variety $\overline{X_\sym} \subset S^d(\R^2)$ is the vanishing locus of
        \[ 
        \frac{\Res_z\left(F(z), z^d F(-1/z) \right)}{F(\ri)F(-\ri)}
        \]
        where $F(z) = \sum_{k=0}^d \left( (k+1) \binom{d}{k+1}t_{k+1} - (d-k+1)\binom{d}{k-1}t_{k-1}\right)z^k$ and $t_{-1} = t_{d+1} = 0$.
    \end{theorem}

    For example, for $d=3$, $\overline{X_\sym} \subset S^{3}(\R^2)$ is the vanishing locus of the quadric
    \[
    t_{1}^{2}-t_{0}t_{2}+t_{2}^{2}-t_{1}t_{3},
    \] the odeco equation \cite{robeva2016orthogonal}. In this case, $X_\sym = \overline{X_\sym}$ \cite{boralevi2017orthogonal}. For $d=4$, $\overline{X_\sym} \subset S^{4}(\R^2)$ is the vanishing locus of 
    \[
    2\,t_{1}^{3}-3\,t_{0}t_{1}t_{2}+t_{0}^{2}t_{3}+2\,t_{1}^{2}t_{3}-3\,t_{0}t_{2}t_{3}-2\,t_{1}t_{3}^{2}-2\,t_{3}^{3}+t_{0}t_{1}t_{4}+3\,t_{1}t_{2}t_{4}-t_{0}t_{3}t_{4}+3\,t_{2}t_{3}t_{4}-t_{1}t_{4}^{2}.
    \]
    We find that $\overline{X_\sym} \subset S^5(\R^2)$ is defined by a quartic with 43 monomials, and $\overline{X_\sym} \subset S^6(\R^2)$ is defined by a quintic with 132 monomials.
    
    We identify symmetric tensors in $S^d(\R^2)$ with homogeneous binary forms of degree $d$
    via 
    \[
    \T \leftrightarrow f_\T(x,y) = \T\left(\begin{pmatrix} x, y\end{pmatrix}, \cdots,\begin{pmatrix} x, y\end{pmatrix}
    \right) = \sum_{k=0}^d \binom{d}{k}t_k x^{d-k} y^k.
    \]
    We have $f = \sum_{k=0}^d a_k x^{d-k} y^k\in V_\sym$ if and only if $a_1 = a_{d-1} = 0$.
    Define $R \coloneqq \mathbb{Q}[a_0,\ldots, a_d]$ for indeterminates $a_0,\ldots, a_d$. We thank Carlos D'Andrea for helping us prove the following, which describes the structure of the numerator in the expression in \Cref{thm:variety-binary-symmetric}.
\begin{lemma} \label{lem:factors-resultant}
Let $d \geq 2$. The resultant of $f(x)=\sum_{k=0}^d a_k x^k$ and $g(x)=x^df(-1/x)$ factors in irreducibles in $R$ as
$\Res_x(f(x),g(x))=(-1)^dp^2\,q, $
where $q \coloneqq f(\ri)f(-\ri)$.
\end{lemma}

\begin{proof}
First, we show that $q$ is irreducible in $R$. Note that $q=f(\ri) f(-\ri)$,
is a factorization of $q$ in irreducibles in $R[\ri]$. Indeed, $q \in R$ is homogeneous of degree $2$ and $f(\ri), f(-\ri) \in R[i] \setminus R$ are homogeneous of degree $1$ in $a_0, \dots, a_d$.

Without loss of generality, let $a_d=1$. Let $\mathbb{K}$ denote the algebraic closure of the field of fractions of $R$, and let $\xi_1,\ldots, \xi_d\in \mathbb{K}$ be the roots of $f(x)$. Then 
$-1/\xi_1,\ldots, -1/\xi_d$ are the roots of $g(x).$ Using the product formula for the resultant (see, e.g., \cite{Gelfand1994DiscriminantsRA})
we obtain
\begin{equation*}\label{pq}
\Res_x(f(x),g(x))=a_0^d\prod_{k=1}^d\prod_{j=1}^d(\xi_k+{1/\xi_j})=(-1)^{d^2}\prod_{k=1}^d\prod_{j=1}^d(\xi_k\xi_j+1),
\end{equation*}
where the second equality follows using $\prod_j \xi_j = (-1)^da_0$.
Collecting factors with $k=j$,
\[
\prod_{j=1}^d(\xi_j^2+1)= \Res_x(f(x), x^2+1)=f(\ri) f(-\ri)=q.
\]
Hence $\Res_x(f(x),g(x)) = p^2 q$, where 
$p \coloneqq \prod_{1\leq j<k\leq d}(\xi_k\xi_j+1)$.
Since $p$ is invariant under permuting the $\xi_j$, we deduce that $p\in R$. We show that $p$ is irreducible.
Suppose $p = p_1 p_2$ with  $p_1, p_2 \in R$ non-constant. Then $p_1$ is
invariant under permuting the $\xi_j$'s.
Hence, if one $(\xi_k\xi_j+1)$ divides $p_1$, all factors of this form must. Hence $p_2$ is constant, a contradiction.
\end{proof} 
    
We are now ready to prove \Cref{thm:variety-binary-symmetric}. Testing whether $\T \in X_\sym \subset S^d(\R^2)$ reduces to computing the resultant of two univariate polynomials, since $\dim \rSO(2) = 1$.

\begin{proof}[Proof of \Cref{thm:variety-binary-symmetric}]
Given $f(u,v) = \sum_{k=0}^d a_k u^{d-k} v^k$, we have that $f \in X_\sym$ if there exists an orthogonal change of variables
\begin{equation}\label{eq:orthogonal-change-variables}
    \begin{pmatrix}
    x \\ y
\end{pmatrix} = 
\frac{1}{\sqrt{1 + z^2}}\begin{pmatrix}
    1 & -z \\ z & 1
\end{pmatrix} \begin{pmatrix}
    u \\ v
\end{pmatrix}
\end{equation}
such that
$f(x,y) \in V_\sym$.
Then $f \in X_\sym$ if there exists $z \in \R$ such that $b_1(z) = b_{d-1}(z) = 0$,
where $f(u - zv, zu + v) = \sum_{k=0}^d b_k(z) u^{d-k}v^k$ for $b_k(z) \in \mathbb{Q}[a_0, \dots, a_d][z]$.
When $z^2 + 1 = 0$, the linear transformation \eqref{eq:orthogonal-change-variables} is not invertible.
So we are interested in $z \neq \ri$. 
We express the coefficients of $b_1(z)$ and $b_{d-1}(z)$ as functions of $a_0, \dots, a_d$, as follows:
\begin{align*}
    (u - zv)^{d-k}(zu + v)^k &= \sum_{i=0}^k \sum_{j=0}^{d-k} \binom{k}{i} \binom{d-k}{j} (-1)^j z^{k-i+j} u^{d-i-j} v^{i+j} = \\ 
    & = \sum_{i=0}^k \sum_{j=0}^{d-k} \binom{k}{i} \binom{d-k}{j} (-1)^{d-k-j} z^{d-k+i-j} u^{i + j} v^{d-i-j}
\end{align*}
The coefficient of $u^{d-1}v$ in $(u - zv)^{d-k}(zu + v)^k$ is $z^{k-1}(k - (d-k)z^2)$, and the coefficient of $uv^{d-1}$ in $(u - zv)^{d-k}(zu + v)^k$ is $(-z)^{d-k-1} \left( (d-k) - k z^2 \right)$.
Hence,
\[
b_1(z) = \sum_{k=0}^d a_k (k - (d-k)z^2)z^{k-1} = \sum_{k=0}^d \left( (k+1) a_{k+1} - (d-k+1)a_{k-1}\right)z^k
\]
and
\begin{align*}
    b_{d-1}(z) &= \sum_{k=0}^d a_k (-z)^{d-k-1} ((d-k) - kz^2) = (-z)^d a_k (1/z)^{k+1} ((d-k) - k (1/z)^{-2}) \\
    &=(-z)^d a_k (1/z)^{k-1} ((d-k)(1/z)^{2} - k ) = -(-z^d) b_1(-1/z).
\end{align*}
Set $F(z) \coloneqq b_1(z)$.
Varying $z \in \R$ in \eqref{eq:orthogonal-change-variables} we parametrize $\rSO(2) \setminus \left\{ \begin{pmatrix}
        0 & \pm 1 \\
        \mp 1 & 0
    \end{pmatrix}\right\}$. 
    Removing these special matrices does not change the Zariski closure:
    \[
    \overline{X_\sym} = \overline{ \rSO(2) \bullet V_\sym } = \overline{ \left\{ \frac{1}{\sqrt{1 + z^2}}\begin{pmatrix}
    1 & z \\ -z & 1
\end{pmatrix} \bullet \T \mid z \in \R, \T \in V_\sym \right\} }.
    \]
Evaluating $F$ at $\ri$ gives
\[
F(\ri) = d \sum_{k = 0}^{\floor{(d-1)/2}} (-1)^k a_{2k+1} + id\sum_{k = 0}^{\floor{d/2}} (-1)^{k+1} a_{2k}.
\]
Hence $b_1(\ri) = b_{d-1}(\ri) = 0$ if and only if 
\[
F(\ri)F(-\ri) = d^2 \left(\left( \sum_{k=0}^{\floor{d/2}} (-1)^k a_{2k}\right)^2 + \left( \sum_{k=0}^{\floor{(d-1)/2}} (-1)^k a_{2k+1}\right)^2 \right) = 0.
\]
The product $F(\ri)F(-\ri)$ is an irreducible factor of $\Res_z(F(z), z^dF(1/z))$ with multiplicity one,
by \Cref{lem:factors-resultant} applied to $F(z) \in \mathbb{Q}[a_0, \dots, a_d][z]$.

If $d$ is odd, then $F$ is the general polynomial of degree $d$ and so
$
r \coloneqq \frac{\Res_z(F(z), z^dF(1/z))}{(F(i)F(-i))}
$ is a perfect square that defines $X_\sym$, by~\Cref{lem:factors-resultant}.

Note that $F$ is not the general polynomial for even $d$, so we cannot directly say that $r$ is a perfect square of an irreducible polynomial. Nevertheless, $r$ defines the polynomials $F$ which share a root with $z^d F(-1/z)$ other than $\ri$, by~\Cref{lem:factors-resultant}. When $z \in \C \setminus \{ \pm \ri \}$ we are parametrizing $\rSO(2, \C)$ in \eqref{eq:orthogonal-change-variables}. Let $X_\sym(\C)$ and $V_\sym(\C)$ be the complexifications of $X_\sym$ and $V_\sym$, respectively. Then, $X_\sym(\C) = \rSO(2, \C) \cdot V_\sym(\C)$, following the argument from  \Cref{prop:Xirredicible}. Therefore, $X_\sym (\C)$ is irreducible and it is the vanishing locus of $r$, so $r$ is a perfect square of an irreducible polynomial by \Cref{lem:factors-resultant}.
\end{proof}

\begin{corollary}
    For $d \geq 3$, the degree of the variety $\overline{X_\sym} \subset S^d(\R^2)$ is $d-1$.
\end{corollary}
\begin{proof}
    The resultant of two degree-$d$  polynomials has degree $2d$. The variety $\overline{X_\sym}$ is the vanishing locus of $r = \frac{\Res_z\left(F(z), z^d F(-1/z) \right)}{F(\ri)F(-\ri)}$, where $F$ is the degree $d$ polynomial in \Cref{thm:variety-binary-symmetric}. So $r$ has degree $2(d-1)$. It is the square of an irreducible polynomial, by \Cref{thm:variety-binary-symmetric} and \Cref{lem:factors-resultant}.
    So $\overline{X_\sym}$ is defined by an irreducible polynomial of degree $d-1$.
\end{proof}

\section{Structured Tucker decompositions}\label{sec:multilinear-svd}

The set $V$ generalizes diagonal matrices in the case $d=2$. We propose a tensor decomposition in which $V$ plays the role as the set of diagonal matrices in the SVD. This decomposition specializes the Tucker decomposition~\cite{tucker1966some}, see~\cite[Chapter 8]{hackbusch2012tensor}, and generalizes tensor diagonalization~\cite{robeva2016orthogonal,robeva2017singular}. 
Structured core tensors in Tucker decompositions also appear in the study of signature tensors~\cite{amendola2019varieties,pfeffer2019learning}. 

\begin{theorem}\label{thm:multilinear-svd}
    If $\T \in \R^{n_1 \times \cdots \times n_d}$ has an orthogonal basis of singular vector tuples then $\T= \Q \cdot \cS$, where $\Q\in \rO(n_1) \times \cdots \times \rO(n_d)$ and $\cS \in V$ with $\cS_{1,\dots, 1}\geq \cdots \geq \cS_{n_1,\dots, n_1} \geq 0$. The 
    first $n_1$ columns of $Q_1, \dots, Q_d$ give the orthogonal singular vectors of $\T$; the diagonal entries of $\cS$ are the corresponding singular values.
    If $\T$ is generic in $X$, 
    the leading subtensor of~$\cS$ indexed by $[n_1]^d$ is unique and the first $n_1$ columns of $Q_1, \dots, Q_d$ are unique up to 
sign.
\end{theorem}

\begin{proof}
Following~\Cref{thm:main}, it remains to show that we can permute the columns of the $Q_k$'s so that $\cS_{1, \dots, 1} \geq \cdots \geq \cS_{n_1, \dots, n_1}$, and we can flip signs to ensure that the diagonal entries are non-negative. Let $\T = \Q \cdot \cS \in X$, with $\Q \in \rO(n_1) \times \cdots \times \rO(n_d)$ and $\cS \in V$. 
    Consider a permutation $\sigma \in S_{n_1}$ such that $|\cS_{\sigma(1), \dots, \sigma(1)}| \geq \cdots \geq |\cS_{\sigma(n_1), \dots, \sigma(n_1)}| \geq 0$, and let~$P_\sigma$ be the corresponding $n_1 \times n_1$ permutation matrix.
    Let $D \in \R^{n_1 \times n_1}$ be a diagonal matrix such that $D_{ii} = 1$ if $S_{\sigma(i), \dots, \sigma(i)} \geq 0$ and $D_{ii} = -1$ if $S_{\sigma(i), \dots, \sigma(i)} < 0$.
    Let $\bP \in \rO(n_1) \times \cdots \times \rO(n_d)$ such that $P_1 = D P_{\sigma}$ and the top-left submatrix of $P_k$ if $P_\sigma$ for all $k \geq 2$. Let $\tilde{\cS} = \bP \cdot \cS$ and~$\tilde{\Q} = (Q_1 P_1^\top, \dots, Q_d P_d^\top)$. Then $\T = \tilde{\Q} \cdot \tilde{\cS}$ and $\tilde{\cS} \in V$ satisfies $\tilde{\cS}_{1 \dots, 1} \geq \cdots \geq \tilde{\cS}_{n_1, \dots, n_1} \geq 0$. If~$\T$ is generic in $X$ then $\tilde{\cS}_{1 \dots, 1} > \cdots > \tilde{\cS}_{n_1, \dots, n_1} > 0$.
\end{proof}

\begin{remark}
    If $d = 2$, then $\cS$ is a diagonal matrix and uniqueness of the first $n_1$ columns of $(Q_1, Q_2)$ up to sign recovers the standard singular value decomposition.
\end{remark}

For symmetric tensors, we get an analogous statement that generalizes the spectral eigendecomposition to higher-order tensors. The proof is similar to the one above, except that for symmetric tensors we act with the same orthogonal matrix on all factors. We include the proof for the symmetric case below for completeness.

\begin{theorem}\label{thm:multilinear-eigendecomposition}
    If $\T \in S^d(\R^n)$ has an orthogonal basis of eigenvectors then $\T= Q \bullet \cS$, where $Q\in \rO(n)$ and $\cS\in V_\sym$ with $\cS_{1,\dots, 1}\geq \cdots \geq \cS_{n,\dots, n}$. The columns of $Q$ are the orthogonal eigenvectors of $\T$; the diagonal entries of $\cS$ are the corresponding eigenvalues.
    If $\T$ is generic and $d$ is even, the decomposition is unique up to the signs of the columns of $Q$. If $d$ is odd, we get generic uniqueness after imposing that the diagonal entries of $\cS$ are non-negative.
\end{theorem}

\begin{proof}
    Let $\T = Q\bullet \cS \in X_\sym$ with $Q\in \rO(n)$ and $\cS\in V_\sym$. We can replace $\cS$ with $\tilde{\cS} = P\bullet \cS$ and $Q$ with $\tilde{Q} = QP^\top$ where $P$ is a permutation matrix so that $\tilde{\cS}_{1, \cdots, 1}\geq \cdots \geq \tilde{\cS}_{n, \dots,  n}$. If $d$ is odd, we can replace $\tilde{\cS}$ with $\hat{\cS} = D\bullet \tilde{\cS}$ and $\hat{Q} = \tilde{Q}D$ for some diagonal matrix with $\pm 1$ on the diagonal so that $\hat{\cS}$ has non-negative diagonal entries. 
    If $\T$ is generic in $X_\sym$ then $\tilde{\cS}_{1, \dots, 1} > \cdots > \tilde{\cS}_{n, \dots, n}$, and the statement follows from \Cref{thm:main-symmetric}.
\end{proof}

\begin{remark}
If $d=2$ then $\cS$ is a diagonal matrix and uniqueness of $Q$ up to the signs of its columns recovers the standard spectral decomposition for symmetric matrices.
\end{remark}

The decomposition in \Cref{thm:multilinear-svd} exists for tensors in $X_\sym$. 
It is also not as restrictive as the odeco decomposition, for which $\cS$ is diagonal. As a consequence, the approximation of an arbitrary tensor $\cT\approx \Q \cdot \cS$ for some $\cS\in V$ will be at least as close as its odeco approximation. Let $\cI = \{ \bi \in [n_1] \times \cdots \times [n_d] \mid \exists j \in [n_1] \text{ s.t. } d_H(\bi, (j, \dots, j)) = 1\}$. An approximation can be found by
\[
{\rm minimize} \;\;\;{\rm dist}(\Q^\top \cdot  \T,V)^2\;=\;\sum_{\bi \in \cI}[\Q^\top \cdot \T]_{\bi}^2\;\;\;\mbox{subject to } \Q \in \rO(n_1) \times \cdots \times \rO(n_d).
\]
Similarly, the decomposition in \Cref{thm:multilinear-eigendecomposition} exists for tensors in $X_\sym$, and is not as restrictive as the symmetric odeco decomposition, for which $\cS$ is diagonal. An approximation of an arbitrary tensor $\cT \approx Q \bullet \cS$ for some $\cS\in V_\sym$ can be found by
\[
{\rm minimize} \;\;\;{\rm dist}(Q^\top \bullet \T,V_\sym)^2\;=\;\sum_{i\neq j}[Q^\top \bullet \T]_{ij\cdots j}^2\;\;\;\mbox{subject to }Q^\top\in \rO(n).
\]

\section{Numerical experiments} \label{sec:numerical}

In this section, we perform numerical computations of the structured Tucker decompositions. We use Riemannian gradient descent to solve the optimization problems at the end of \Cref{sec:multilinear-svd}. We do so using \texttt{Pymanopt} \cite{townsend2016pymanopt}. We compare to (symmetric) odeco tensors, replacing $V$ (resp. $V_{\sym}$) with $V_{\diag}$ and using the same optimization approach.

\subsection{Cumulant tensors} \label{sec:numerical-CA} Given an $n$-dimensional random vector $Z$, let $K_d(Z) \in S^{d}(\R^n)$ be the tensor of all order-$d$ cumulants of $Z$. We analyze how well the symmetric odeco tensors, $X_{\sym-\mathrm{odeco}}$, and our set $X_{\sym}$ approximate the cumulant tensors of real datasets. This is closely related to independent and mean independent component analysis \cite{mesters2022non}. In practice, we do not have access to the exact cumulants of $Z$, so we estimate them using k-statistics (see, e.g., \cite[Chapter 4]{mccullagh2018tensor}) with \texttt{PyMoment} \cite{PyMoments:2020}.

The Iris flower dataset \cite{iris_53} contains 150 measurements of iris flowers, each with four features (sepal and petal length and width). The Fama-Fench Data Library, provides financial datasets. We analyze the average value weighted returns of 25 portfolios along 1186 timesteps\footnote{We use the dataset called ``25 Portfolios Formed on Size and Book-to-Market (5 x 5)'' available at \texttt{\url{https://mba.tuck.dartmouth.edu/pages/faculty/ken.french/data_library.html}}.}.
The electroencephalogram (EEG) dataset \cite{chavarriaga2010learning} records the brain's electrical activity of six subjects over two sessions. We analyze 91648 observations of 64 electrodes\footnote{We use the data indexed by S01-1 in the 22nd dataset available at 
\texttt{\url{https://bnci-horizon-2020.eu/database/data-sets}}.}.

For Iris and Portfolios, we compute the fourth-order cumulant tensor. For EEG, we compute its third-order cumulant tensor. For a tensor $\T$, we define the distance to $X_{\sym}$ as $\min_{Q \in \rO(n)}\dist(Q \bullet \T, V_{\sym})/ \|\T\|$, where we divide by $\| \T\|$ so that the result is in between zero and one. We tackle this optimization problem via Riemannian gradient descent, initializing at 20 random points. For $X_{\sym - \mathrm{odeco}}$ we replace $V_{\sym}$ with $V_{\diag}$. \Cref{tab:distance-X-odeco-symmetric} shows that $X_{\sym}$ approximates the cumulant tensors better than $X_{\sym - \mathrm{odeco}}$.

\begin{table}[h]
    \centering
    \begin{tabular}{l|c|c|c}
         Tensor & Format & Distance to $X_{\sym}$ & Distance to $X_{\sym-\mathrm{odeco}}$ \\
         \hline
        Iris & $4 \times 4 \times 4 \times 4$ & $0.23$ & $0.67 \vphantom{a^{\frac{a^a}{a^a}}}$ \\
        Portfolios & $25 \times 25 \times 25 \times 25$ & $0.036$ & $0.36$ \\
        EEG & $64 \times 64 \times 64$ & $0.068$ & $0.60$\\
    \end{tabular}
    \caption{Comparison of how well $X_{\sym}$ and $X_{\sym-\mathrm{odeco}}$ approximate each cumulant tensor}
    \label{tab:distance-X-odeco-symmetric}
\end{table}

\subsection{Data tensors} \label{sec:numerical-data-tensors}
We analyze how well the varieties $X$ and $X_{\mathrm{odeco}}$ approximate data tensors\footnote{These three tensors are available at 
\texttt{\url{https://gitlab.com/tensors}}.}. The monkey brain machine interface (BMI) tensor \cite{vyas2018neural, vyas2020causal} contains spike data from $43$ neurons along $200$ time steps in $88$ trials. The excitation emission matrix (EEM) tensor \cite{acar2014structure} contains fluorescent spectroscopy measurements of $18$ samples of $251$ emissions and $21$ excitations. The Miranda turbulent flow tensor \cite{cabot2006reynolds, zhao2020sdrbench} measures the density of a mixing of two fluids on a $2048 \times 256 \times 256$ grid. 
As before, we scale the distance to $X$ or $X_{\mathrm{odeco}}$ by dividing by the norm of the tensor, so that the result is between zero and one. We report the minimum value obtained from running Riemannian gradient descent with 20 random initial points, both for $X$ and $X_{\mathrm{odeco}}$. \Cref{tab:distance-X-odeco} shows that we approximate the data tensors better with $X$ than with $X_{\mathrm{odeco}}$.

\begin{table}[h]
    \centering
    \begin{tabular}{l|c|c|c}
         Tensor & Format & Distance to $X$ & Distance to $X_\mathrm{odeco}$ \\
         \hline
        BMI & $43 \times 200 \times 88$ & $2.1 \times 10^{-3}$ & $4.7 \times 10^{-1} \vphantom{a^{\frac{a^a}{a^a}}}$ \\
        EEM & $18 \times 251 \times 21$ & $4.0 \times 10^{-3}$ & $3.5 \times 10^{-1}$ \\
        Miranda & $2048 \times 256 \times 256$ & $6.0 \times 10^{-3}$ & $1.2 \times 10^{-1}$\\
    \end{tabular}
    \caption{Comparison of how well $X$ and $X_{\mathrm{odeco}}$ approximate each tensor}
    \label{tab:distance-X-odeco}
\end{table}

\begin{remark}
    In the experiments in Sections \ref{sec:numerical-CA} and \ref{sec:numerical-data-tensors} we reported the optimum value of 20 trials, since we focus on goodness of fit. However, the \texttt{RGD} approaches are fairly stable under different initializations: the standard deviation of the objective function at convergence for the different trials is at least one order of magnitude smaller than its mean.
\end{remark}

\section*{Discussion}
We have shown that generic (symmetric) tensors with an orthogonal basis of singular vector tuples (resp. eigenvectors) have a unique such basis, except for $2 \times 2 \times 2 \times 2$ tensors. For matrices, generic here means that all its singular values (resp. eigenvalues) are distinct. Obtaining genericity conditions in this context for tensors in $X$ (resp. $X_\sym$) is an interesting direction of future work. Moreover, there are numerical and optimization challenges arising from the tensor decompositions we propose in this paper. We think that analyzing the convergence guarantees of the Riemannian gradient descent approach or coming up with new algorithms to decompose tensors in $X$ or $X_\sym$ deserves further investigation.

\section*{Acknowledgments}
We thank Carlos D'Andrea, Nathan Henry, and Geert Mesters  for helpful discussions. AR was supported by fellowships from ``la
Caixa” Foundation (ID 100010434), with fellowship code LCF/BQ/EU23/12010097, and from RCCHU.
AS was supported by an Alfred P. Sloan research fellowship.
PZ was
supported by NSERC grant RGPIN-2023-03481.  

\bibliographystyle{alpha}
\bibliography{biblio}

\newcommand{\etalchar}[1]{$^{#1}$}
\begin{thebibliography}{DLDMV00}

\bibitem[AFS19]{amendola2019varieties}
Carlos Am{\'e}ndola, Peter Friz, and Bernd Sturmfels.
\newblock Varieties of signature tensors.
\newblock In {\em Forum of Mathematics, Sigma}, volume~7, page e10. Cambridge University Press, 2019.

\bibitem[AGH{\etalchar{+}}14]{anandkumar2014tensor}
Animashree Anandkumar, Rong Ge, Daniel~J Hsu, Sham~M Kakade, and Matus Telgarsky.
\newblock Tensor decompositions for learning latent variable models.
\newblock {\em J. Mach. Learn. Res.}, 15(1):2773--2832, 2014.

\bibitem[APG{\etalchar{+}}14]{acar2014structure}
Evrim Acar, Evangelos~E Papalexakis, G{\"o}zde G{\"u}rdeniz, Morten~A Rasmussen, Anders~J Lawaetz, Mathias Nilsson, and Rasmus Bro.
\newblock Structure-revealing data fusion.
\newblock {\em BMC bioinformatics}, 15:1--17, 2014.

\bibitem[BDHR17]{boralevi2017orthogonal}
Ada Boralevi, Jan Draisma, Emil Horobe{\c{t}}, and Elina Robeva.
\newblock Orthogonal and unitary tensor decomposition from an algebraic perspective.
\newblock {\em Israel journal of mathematics}, 222:223--260, 2017.

\bibitem[CC06]{cabot2006reynolds}
William~H Cabot and Andrew~W Cook.
\newblock Reynolds number effects on rayleigh--taylor instability with possible implications for type ia supernovae.
\newblock {\em Nature Physics}, 2(8):562--568, 2006.

\bibitem[CM10]{chavarriaga2010learning}
Ricardo Chavarriaga and Jos{\'e} del~R Mill{\'a}n.
\newblock Learning from eeg error-related potentials in noninvasive brain-computer interfaces.
\newblock {\em IEEE transactions on neural systems and rehabilitation engineering}, 18(4):381--388, 2010.

\bibitem[DLDMV00]{lathauwerSVD}
Lieven De~Lathauwer, Bart De~Moor, and Joos Vandewalle.
\newblock A multilinear singular value decomposition.
\newblock {\em SIAM Journal on Matrix Analysis and Applications}, 21(4):1253--1278, 2000.

\bibitem[DOT18]{draisma2018best}
Jan Draisma, Giorgio Ottaviani, and Alicia Tocino.
\newblock Best rank-k approximations for tensors: generalizing {E}ckart--{Y}oung.
\newblock {\em Research in the Mathematical Sciences}, 5(2):27, 2018.

\bibitem[Fis36]{iris_53}
R.~A. Fisher.
\newblock {Iris}.
\newblock UCI Machine Learning Repository, 1936.
\newblock {DOI}: https://doi.org/10.24432/C56C76.

\bibitem[GKZ94]{Gelfand1994DiscriminantsRA}
Israel~M. Gelfand, Mikhail~M. Kapranov, and Andrei~V. Zelevinsky.
\newblock {\em Discriminants, Resultants, and Multidimensional Determinants}.
\newblock Birkhäuser Boston, MA, 1994.

\bibitem[Hac12]{hackbusch2012tensor}
Wolfgang Hackbusch.
\newblock {\em Tensor spaces and numerical tensor calculus}, volume~42.
\newblock Springer, 2012.

\bibitem[Kol01]{kolda2001orthogonal}
Tamara~G Kolda.
\newblock Orthogonal tensor decompositions.
\newblock {\em SIAM Journal on Matrix Analysis and Applications}, 23(1):243--255, 2001.

\bibitem[Lim05]{lim2005singular}
Lek-Heng Lim.
\newblock Singular values and eigenvalues of tensors: a variational approach.
\newblock In {\em 1st IEEE International Workshop on Computational Advances in Multi-Sensor Adaptive Processing, 2005.}, pages 129--132. IEEE, 2005.

\bibitem[McC18]{mccullagh2018tensor}
Peter McCullagh.
\newblock {\em Tensor methods in statistics: Monographs on statistics and applied probability}.
\newblock Chapman and Hall/CRC, 2018.

\bibitem[Mum04]{mumford2004red}
David Mumford.
\newblock {\em The red book of varieties and schemes: includes the Michigan lectures (1974) on curves and their Jacobians}, volume 1358.
\newblock Springer, 2004.

\bibitem[MZ24]{mesters2022non}
Geert Mesters and Piotr Zwiernik.
\newblock Non-independent components analysis.
\newblock {\em The Annals of Statistics}, 52(6):2506--2528, 2024.

\bibitem[PSS19]{pfeffer2019learning}
Max Pfeffer, Anna Seigal, and Bernd Sturmfels.
\newblock Learning paths from signature tensors.
\newblock {\em SIAM Journal on Matrix Analysis and Applications}, 40(2):394--416, 2019.

\bibitem[QCC18]{qi2018tensor}
Liqun Qi, Haibin Chen, and Yannan Chen.
\newblock {\em Tensor eigenvalues and their applications}, volume~39.
\newblock Springer, 2018.

\bibitem[Qi05]{qi2005eigenvalues}
Liqun Qi.
\newblock Eigenvalues of a real supersymmetric tensor.
\newblock {\em Journal of Symbolic Computation}, 40(6):1302--1324, 2005.

\bibitem[RHST24]{ribot2024decomposing}
Alvaro Ribot, Emil Horobet, Anna Seigal, and Ettore~Teixeira Turatti.
\newblock Decomposing tensors via rank-one approximations.
\newblock {\em arXiv:2411.15935}, 2024.

\bibitem[Rob16]{robeva2016orthogonal}
Elina Robeva.
\newblock Orthogonal decomposition of symmetric tensors.
\newblock {\em SIAM Journal on Matrix Analysis and Applications}, 37(1):86--102, 2016.

\bibitem[Rom05]{roman2005advanced}
Steven Roman.
\newblock {\em Advanced linear algebra}, volume~3.
\newblock Springer, 2005.

\bibitem[RS17]{robeva2017singular}
Elina Robeva and Anna Seigal.
\newblock Singular vectors of orthogonally decomposable tensors.
\newblock {\em Linear and Multilinear Algebra}, 65(12):2457--2471, 2017.

\bibitem[Smi20]{PyMoments:2020}
Kevin~D. Smith.
\newblock Pymoments: A python toolkit for unbiased estimation of multivariate statistical moments.
\newblock \texttt{https://github.com/KevinDalySmith/PyMoments}, 2020.

\bibitem[TKW16]{townsend2016pymanopt}
James Townsend, Niklas Koep, and Sebastian Weichwald.
\newblock Pymanopt: A python toolbox for optimization on manifolds using automatic differentiation.
\newblock {\em Journal of Machine Learning Research}, 17(137):1--5, 2016.

\bibitem[Tuc66]{tucker1966some}
Ledyard~R Tucker.
\newblock Some mathematical notes on three-mode factor analysis.
\newblock {\em Psychometrika}, 31(3):279--311, 1966.

\bibitem[VECS{\etalchar{+}}18]{vyas2018neural}
Saurabh Vyas, Nir Even-Chen, Sergey~D Stavisky, Stephen~I Ryu, Paul Nuyujukian, and Krishna~V Shenoy.
\newblock Neural population dynamics underlying motor learning transfer.
\newblock {\em Neuron}, 97(5):1177--1186, 2018.

\bibitem[VNVM14]{vannieuwenhoven2014generic}
Nick Vannieuwenhoven, Johannes Nicaise, Raf Vandebril, and Karl Meerbergen.
\newblock On generic nonexistence of the {S}chmidt--{E}ckart--{Y}oung decomposition for complex tensors.
\newblock {\em SIAM Journal on Matrix Analysis and Applications}, 35(3):886--903, 2014.

\bibitem[VORS20]{vyas2020causal}
Saurabh Vyas, Daniel~J O’Shea, Stephen~I Ryu, and Krishna~V Shenoy.
\newblock Causal role of motor preparation during error-driven learning.
\newblock {\em Neuron}, 106(2):329--339, 2020.

\bibitem[ZDL{\etalchar{+}}20]{zhao2020sdrbench}
Kai Zhao, Sheng Di, Xin Lian, Sihuan Li, Dingwen Tao, Julie Bessac, Zizhong Chen, and Franck Cappello.
\newblock Sdrbench: Scientific data reduction benchmark for lossy compressors.
\newblock In {\em 2020 IEEE international conference on big data (Big Data)}, pages 2716--2724. IEEE, 2020.

\bibitem[ZG01]{zhang2001rank}
Tong Zhang and Gene~H Golub.
\newblock Rank-one approximation to high order tensors.
\newblock {\em SIAM Journal on Matrix Analysis and Applications}, 23(2):534--550, 2001.

\end{thebibliography}

\end{document}